\documentclass{article}

\usepackage{amsmath,amsfonts,amssymb,amsthm}
\usepackage{epsfig}
\usepackage{mathrsfs}
\usepackage{enumerate}
\newtheorem{theorem}{Theorem}
\newtheorem{lemma}[theorem]{Lemma}
\newtheorem{koro}[theorem]{Corollary}
\newtheorem{thm}{Theorem}[section]
\newtheorem{lem}[thm]{Lemma}

\newtheorem{prop}[thm]{Proposition}

\newcommand{\A}[1]{\vspace{6mm}}

\newcommand{\bd}{{\mathrm{bd}}\,}

\newcommand{\relbd}{{\mathrm{relbd}}\,}

\newcommand{\lin}{{\mathrm{lin}}\,}

\newcommand{\ee}{\varepsilon}

\newcommand{\R}{{\mathbb R}}

\newcommand{\N}{{\mathbb N}}

\newcommand{\cH}{{\mathcal  H}}

\newcommand{\cF}{{\mathcal  F}}
\newcommand{\cK}{{\mathcal  K}}
\newcommand{\cP}{{\mathcal  P}}

\newcommand{\bK}{{\mathbb K}}
\newcommand{\Ha}{{\mathcal H}}
\newcommand{\nor}{\mathrm{nor}\,}

\renewcommand{\j}{j}

\catcode`@=11
\def\section{%
\setcounter{equation}{0} \setcounter{theorem}{0} \@startsection
{section}{1}{\z@}{-4.0ex plus -1ex minus
    -.2ex}{2.3ex plus .2ex}{\bf}}
\catcode`@=12 \theoremstyle{definition}

\begin{document}

\title{A flag representation of projection functions}

\author{Paul Goodey, Wolfram Hinderer, Daniel Hug,\\ Jan Rataj and 
Wolfgang Weil}
\date{\today}

\maketitle
\begin{abstract}
\noindent 
The $k$th projection function $v_k(K,\cdot)$ of a convex body $K\subset \R^d, d\ge 3,$ is a function on the Grassmannian $G(d,k)$ which measures the $k$-dimensional volume of the projection of $K$ onto members of $G(d,k)$. For $k=1$ and $k=d-1$, simple formulas for the projection functions exist. In particular, $v_{d-1}(K,\cdot)$ can be written as a spherical integral with respect to the surface area measure of $K$. Here, we generalize this result and prove two integral representations for $v_k(K,\cdot), k=1,\dots,d-1$, over flag manifolds. Whereas the first representation generalizes a result of Ambartzumian (1987), but uses a flag measure which is not continuous in $K$, the second representation is related to a recent flag formula for mixed volumes by Hug, Rataj and Weil (2013) and depends continuously on $K$. 


\end{abstract}

\section{Introduction}

Let $\mathcal K$ be the space of  convex bodies (non-empty compact convex sets) in $\R^d, d\ge 3,$ supplied with the Hausdorff metric. For $K\in {\mathcal K}$ and $k\in\{0,....,d\}$, the $k$th projection function $v_k(K,\cdot)$ of $K$ is a continuous function on $G(d,k)$, the Grassmannian of $k$-dimensional subspaces in $\R^d$, and is defined by
$$
v_k(K,E) = V_k(K |E),\quad E\in G(d,k).
$$
Here $K|E$ is the orthogonal projection of $K$ onto $E$ and $V_k$ denotes the $k$th intrinsic volume which, for a body in the $k$-dimensional space $E$, equals the ($k$-dimensional) volume in $E$. The projection functions $v_0(K,\cdot)=1$ and $v_d(K,\cdot)=V_d(K)$ are trivial. For $k=1$, $v_1(K|E)$ is the width of $K$ in direction of the line $E$, hence
$$
v_1(K,E) = h(K,u)+h(K,-u) ,
$$
where $h(K,\cdot)$ denotes the support function of $K$ and $u$ is in the direction of the line $E$. 
%Hence $v_1(K,\cdot)$ can be interpreted as an even function on the unit sphere $S^{d-1}$. 

For $k=d-1$, we have a simple and well-known integral representation, namely 
\begin{equation}\label{projfunc}
v_{d-1}(K,x^\bot) = \frac{1}{2} \int_{S^{d-1}} |\langle x,u\rangle |\, S_{d-1}(K, du) , \quad x\in S^{d-1} .
\end{equation}
Here, $S^{d-1}$ is the unit sphere, $\langle\cdot, \cdot\rangle$ denotes the standard scalar product in $\R^d$ and $S_{d-1}(K,\cdot )$ is the $(d-1)$st surface area measure of $K$ (see \cite{S}, for the standard notions from convex geometry which we use). For bodies $K\in{\mathcal K}$ which are centrally symmetric and smooth enough, a corresponding integral representation for arbitrary projection functions is known. This involves the projection generating measure $\rho_k(K,\cdot )$ of $K$, a (signed) measure on $G(d,k)$, and reads
\begin{equation}\label{projfunc2}
v_{k}(K,E) = \frac{2^k}{k!} \int_{G(d,k)} |\langle E,F\rangle | \,\rho_{k}(K, dF) , \quad E\in G(d,k) ,
\end{equation}
where $|\langle E,F\rangle |$ is the absolute determinant of the projection of $E$ onto $F$. \eqref{projfunc2} holds more generally for generalized zonoids $K$, but it is also known that a corresponding formula for all centrally symmetric bodies $K$ cannot hold, for $k\in\{ 1,...,d-2\}$, at least not with a (signed) measure $\rho_k(K,\cdot)$ on $G(d,k)$, (see \cite{GW} and the remarks on p. 635 of \cite{HRW}).

Let $\cal P\subset \cal K$ be the dense subset of convex polytopes. For $k=1$, $d=3$ and $P\in\cal P$, Ambartzumian \cite{Amb1,Amb2} introduced a new concept by showing that the width function $v_1(P,\cdot)$ of $P$ has an integral representation with a certain measure on a flag manifold.  It is a first goal of this paper to generalize this result to arbitrary dimensions $d$ and $k$.
To be more precise, we introduce, for $1\le q\le d$, the {\it flag manifold}
$$
F(d,q) := \{ (u,U) : U\in G(d,q), u\in U\cap S^{d-1}\} .
$$
With the usual topology, $F(d,q)$ is a compact space. For a polytope $P$ and $k\in\{1,...,d-1\}$, we define a Borel measure $\tau_k(P,\cdot )$ on 
$F(d,d-k)$ by 
\begin{equation}\label{firstflagmeasure}
\tau_k(P,\cdot ) := \sum_{F\in{\mathcal F}_k(P)} V_k(F)\, \int_{n(P,F)} {\bf 1}\{ (u,F^\bot)\in \cdot\}\ {\mathcal H}^{d-1-k}(du) .
\end{equation}
Here, ${\mathcal F}_k(P)$ is the collection of $k$-faces of $P$, $n(P,F)$ is the intersection of the normal cone $N(P,F)$ of $P$ at $F$ with the unit sphere $S^{d-1}$ (that is, a $(d-1-k)$-dimensional spherical polytope), $F^\bot\in G(d,d-k)$ is the linear subspace orthogonal to $F$ and ${\mathcal H}^{d-1-k}$ is the Hausdorff measure of dimension $d-1-k$.
For $d=3$, the measure $\tau_1(P,\cdot )$ coincides, up to the factor $2\pi$, with the one introduced by Ambartzumian \cite{Amb1,Amb2}.
 
\begin{theorem}\label{TH1} 
For $1\le k\le d-2$ and  $P\in {\mathcal P}$, we have 
\begin{equation}\label{th1}
v_k(P,E) =  \omega_{d-k}^{-1}\int_{F(d,d-k)} \frac{\langle u^\bot\cap U, u^\bot\cap E^\bot\rangle^2}{\| u | E^\bot \|^{d-k-2}}\ \tau_k(P, d(u,U)),\quad E\in G(d,k).
\end{equation}
\end{theorem} 
Here, $\omega_n$ denotes the surface area of the $n$-dimensional unit ball and the integrand is defined as $0$ for $u\in E$. This result remains true for $k=d-1$ and reproduces in the special case of polytopes equation \eqref{projfunc}, if the integrand is properly interpreted as $\|u| E^\bot \|=|\langle u,v\rangle|$ for a unit vector $v\in E^\bot$.

For $d=3$ and $k=1$, Ambartzumian extended this integral formula for projection functions to arbitrary convex bodies $K\subset \R^3$ by using polytopal approximation. 
Now, let  $d\ge 3$ and $1\le k\le d-2$. With a given convex body $K\in \cal K$, we can associate  an approximating sequence of polytopes $P_i\to K$, $P_i\in{\cal P}$. Since  $\tau_k(P_i,F(d,d-k))= \omega_{d-k}V_k(P_i)$ and $V_k(P_i)\to V_k(K)$, we can choose a weakly convergent subsequence of the measures $\tau_k(P_i,\cdot)$, $i\in\N$, to obtain a limit measure $\tau_k(K,\cdot )$. However, it is important to notice that even for $d=3$ and $k=1$, the integrand in \eqref{th1} is not a continuous function. This indicates that the extension 
of Theorem \ref{TH1} to arbitrary convex bodies $K$ by a limit procedure requires further arguments (see Proposition \ref{PRP2}). Moreover, as was shown by an example in \cite{HHW}, the limit measure $\tau_k(K,\cdot )$ depends on the approximating sequence $(P_i)_{i\in\N}$ and is, thus, not continuous in $K$.  

The second and major goal of this work is therefore concerned with an integral representation of $v_k(K,\cdot )$, in the spirit of \eqref{th1}, but with a flag measure $\psi_k(K,\cdot)$ on $F(d,d-k)$, which depends continuously on $K\in{\mathcal K}$. To be more precise, let $\psi_k(P,\cdot)$, for $P\in{\mathcal P}$, be the measure on $F(d,d-k)$ given by
\begin{align}\label{secondflagmeasure}
\psi_k(P,\cdot )&:=  \sum_{F\in{\mathcal F}_k(P)} V_k(F)\int_{n(P,F)} \int_{G(\langle u\rangle, d-k)} {\bf 1}\{ (u,U)\in \cdot\} \cr
&\qquad\qquad\qquad\qquad \times 
\langle U,F^\bot\rangle^2 \, \nu_{d-k}^{\langle u\rangle} (dU) \,\cH^{d-1-k}(du) .
\end{align}
Here, $G(\langle u\rangle, d-k)$ denotes the Grassmannian of $(d-k)$-flats containing the line $\langle u\rangle$ generated by $u$ and $\nu_{d-k}^{\langle u\rangle} $ is the corresponding invariant probability measure. 
As follows from results in \cite{Hind, HHW}, $\psi_k(P,\cdot )$ satisfies a local Steiner formula for $P$ and  therefore has a continuous extension to all $K\in{\mathcal K}$. 
Our main result is the following.
 
\begin{theorem}\label{TH2} 
For $1\le k\le d-1$, there is a function $g$ on $G(d,k)\times F(d,d-k)$ such that 
\begin{equation}\label{th2}
v_k(K,E) =  \int_{F(d,d-k)} g(E,u,U)\ \psi_k(K, d(u,U))
\end{equation}
holds for all  $K\in {\mathcal K}$, and almost all $E\in G(d,k)$. 
\end{theorem} 

Here `almost all' refers to the invariant probability measure $\nu_k^d$ on $G(d,k)$ and the `exceptional set' may depend on $K$.

The setup of the paper is as follows. In the next section, we introduce the necessary notation and we collect some background information. In Section 3, we give the proof of Theorem \ref{TH1}. The subsequent two sections are concerned with the proof of Theorem \ref{TH2}. The final section discusses some functional analytic aspects concerning an integral equation which connects the two representations \eqref{th1} and \eqref{th2}. In the Appendix we prove an auxiliary lemma and a combinatorial identity, both of which are used in the solution of this integral equation. 

\section{Preliminaries}

In the following, we work in Euclidean space $\R^d, d\ge 3$, with scalar product $\langle \cdot,\cdot\rangle$ and norm $\|\cdot\|$. We denote the unit ball and the unit sphere by $B^d$ and $S^{d-1}$, respectively, and let ${\mathcal H}^j$ be the $j$-dimensional Hausdorff measure. We use the abbreviations $\kappa_d :=\cH^d(B^d)$ and $\omega_d :=d\kappa_d$. For a set $A\subset \R^d$, $\bd A$ is the topological boundary and $\lin A$ is the linear hull. For real numbers $a,b$, we denote the minimum by $a\wedge b$. We also use the symbol $\wedge$ for the exterior product of vectors; the usage will always be clear from the context.
We have already introduced the classes $\cK$ and $\cP$ of convex bodies and convex polytopes. For $P\in {\mathcal P}$, ${\mathcal F}_k(P)$ is the set of $k$-dimensional faces of $P$, $k\in\{ 0,\dots ,d-1\}$. For $P\in \cP$ and a face $F\in\bigcup_{k=0}^{d-1} \cF_k(P)$ of $P$, $N(P,F)$ is the normal cone of $P$ at $F$ and $n(P,F) :=N(P,F)\cap S^{d-1}$. 

The Grassmann manifold $G(d,k)$ is the set of $k$-dimensional linear subspaces of $\R^d$, supplied with its usual topology and with the invariant (probability) measure $\nu_k^d$.  
For subspaces $A,B\in G(d,k)$, $|\langle A,B\rangle |$ is the (absolute value of the) determinant of the orthogonal projection of $A$ onto $B$ (or vice versa) and $\langle A,B\rangle^2$ is the corresponding square value. We shall repeatedly use that $|\langle A,B\rangle |=|\langle A^\bot,B^\bot\rangle |$ for $A,B\in G(d,k)$.  More generally, we need the $i$th product $\langle A,B\rangle_i$, $i=0,\dots ,k\wedge (d-k)$, for which we refer the reader to \cite{HRW}, for a detailed description. Roughly speaking, $\langle A,B\rangle_i$ is the length of the orthogonal projection of a simple unit $k$-vector representing $B$ onto the space of exterior products of unit simple $(k-i)$-vectors in $A$ with corresponding $i$-vectors in the orthogonal complement $A^\bot$. We mention that $\langle A,B\rangle_i=\langle B,A\rangle_i$ and emphasize the special case $\langle A,B\rangle_0=|\langle A,B\rangle|$. 
%For a subspace $E\in G(d,q)$, we sometimes use $p_E$ to denote the orthogonal projection onto $E$.
For a subspace $L\in G(d,k)$ and $j\in\{0,\dots ,d-1\}$, we write $G(L,j)$ for the manifold of subspaces $M\in G(d,j)$ which contain $L$, if $j\ge k$, or are contained in $L$, if $j< k$. We also denote the corresponding invariant probability measure on $G(L,j)$ by $\nu^L_j$.

Measures on flag manifolds have been investigated in recent years for convex bodies and more general sets (for example, sets of positive reach). Since we are interested in projection properties, the concentration on convex sets seems natural. For convex bodies $K$, flag measures can be introduced in various ways, as projection means of support measures, by direct representations on the normal bundle, or by a local Steiner formula, generalizing the classical two local descriptors of convex sets, the area and curvature measures. A survey on flag measures with further historical remarks and references is given in \cite{HTW}. Here, we only need area-type flag measures, as they were studied and used in \cite{HRW} and \cite{HHW}. For $1\le q\le d$, we consider the flag manifold 
$F(d,q)$, defined in the introduction and consisting of pairs $(u,U)$ of a subspace $U\in G(d,q)$ and a unit vector $u$ in $U$. For $p\in \{0,\ldots,d-1\}$, we shall also use the manifold $F^\bot (d, p):=\{(u,V)\in S^{d-1}\times G(d,p):u\in V^\bot\}$, which arises from $F(d,q)$ by the orthogonality map $(u,U)\mapsto (u,U^\bot)$ (thus $p=d-q$). Another simple transformation replaces $F^\perp(d,p)$ by $F^\perp(d,d-1-p)$ via the map $(u,U)\mapsto (u,U^\bot \cap u^\bot)$. 

The two (series of) flag measures which we introduced in the previous section are defined on $F(d,d-k), k\in\{ 1,\dots, d-1\}$. The first, $\tau_k(K,\cdot)$, is given, for polytopes $K$, by \eqref{firstflagmeasure} and leads to the projection formula in Theorem \ref{TH1}. Although this flag measure and the projection formula can be extended to arbitrary bodies $K\in\cK$ by a compactness argument, the extended measure is not unique and thus is not continuous in $K$. Details will be given in the next section. It therefore seems that this flag measure is less interesting for applications in convex geometry. The second measure, $\psi_k(K,\cdot)$, was also defined for polytopes in a direct way, by \eqref{secondflagmeasure}, but it has a (weakly) continuous extension to all convex bodies $K\in\cK$, as follows from the subsequent local Steiner formula (see \cite{HHW, HTW}). For $K\in\cK, j\in\{0,\dots ,d-1\}, \rho > 0,$ and a Borel set $\eta\subset F^\perp(d,j)$, we consider the local parallel set 
$$M^{(j)}_\rho(K,\eta) :=\{E\in A(d,j) : 0<d(K,E)\le\rho, (u(K,E), L(E))\in\eta\} \, .$$
Here, $A(d,j)$ is the affine Grassmannian, $d(K,E)$ the (Euclidean) distance between $K$ and $E$, $u(K,E)$ the direction from $K$ to $E$ in which the distance is attained, and $L(E)\in G(d,j)$ the subspace parallel to $E\in A(d,j)$. Notice that the direction $u(K,E)$ is unique, although it may be realized in more than one pair $(x,y) \in K\times E$. Then, $M^{(j)}_\rho(K,\eta)$ is a Borel set in $A(d,j)$ and its Haar measure $\mu_j^d(M^{(j)}_\rho(K,\eta))$ has a polynomial expansion in $\rho$,
$$
\mu_j^d(M^{(j)}_\rho(K,\eta)) = \frac{1}{d-j}\sum_{m=0}^{d-j-1}\rho^{d-j-m}{{d-j}\choose{m}}S_m^{(j)}(K,\eta) ,
$$
with coefficients $S_m^{(j)}(K,\cdot)$ which are finite (nonnegative) measures on $F^\perp(d,j)$. There are many different ways to formulate this local flag-type Steiner formula and to normalize the measures. Here, we based the result on the normalization of the Haar measure $\mu_j^d$ used in \cite{SW}. The flag measure $S_m^{(j)}(K,\cdot)$  has the property that its image under the mapping $(u,V)\mapsto u$ is a multiple of the $m$th area measure $S_m(K,\cdot)$ of $K$, and the total measure is proportional to the $m$th intrinsic volume of $K$. More precisely,
$$S_m^{(j)}(K,F^\perp(d,j)) =\frac{\omega_{d-j}}{\omega_d}S_m(K,S^{d-1})=\frac{\omega_{d-j}\omega_{d-m}}{ \binom{d-1}{m}\omega_d}V_m(K).$$
For a survey on the various flag measures and their interrelations, see \cite{HTW}.

The measure $\psi_k(K,\cdot)$, defined for polytopes by \eqref{secondflagmeasure}, is closely related to the measure $S_{k}^{(d-1-k)}(K,\cdot)$. In fact, in \cite[p. 21]{HHW} it is shown that
\begin{align}\label{conn}
\int_{F^\perp(d,d-1-k)} &f(u,V+\langle u\rangle)\, S^{(d-1-k)}_k(P,d(u,V)) \\ \nonumber
&= \frac{\omega_{k+1}}{\omega_d}\int_{F(d,d-k)} f(u,U)\, \psi_k(P,d(u,U)) ,
\end{align} 
for all measurable functions $f : F(d,d-k)\to [0,\infty)$. Since $S_{k}^{(d-1-k)}(K,\cdot)$ depends continuously on $K\in{\mathcal K}$, \eqref{conn} shows that $\psi_k(K,\cdot)$ has a continuous extension to all convex bodies $K$ and that
$$
\psi_k(K,F(d,d-k))= \frac{\omega_{d-k}}{\binom{d-1}{k}} V_k(K) .
$$
Moreover, \eqref{conn} then holds for general bodies $K\in\mathcal K$.

For our proof of Theorem \ref{TH2} we use  measure geometric representations of the intrinsic volumes and the flag measures, described in more detail in \cite{HRW}. Namely, for a convex body $K\in\mathcal K$, let
$$\nor K :=\{(x,u)\in\bd K\times S^{d-1}:\, \langle u, y-x\rangle \leq 0\text{ for all }y\in K\}$$
be its unit normal bundle. It is known that at $\cH^{d-1}$-almost all $(x,u)\in\nor K$, there exist principal directions $a_i(K;x,u)\in S^{d-1}$ and associated principal curvatures 
$k_i(K;x,u)\in[0,\infty]$, $i=1,\ldots,d-1$, and we put
$$\bK_I(K;x,u):=\frac{\prod_{i\in I}k_i(K;x,u)}{\prod_{i=1}^{d-1}\sqrt{1+k_i(K;x,u)^2}}$$
and
$$A_I(K;x,u):=\lin\{ a_i(K;x,u):\, i\in I\},$$
whenever $I\subset\{1,\ldots,d-1\}$ (with the appropriate interpretation, if $k_i(K;x,u) =\infty$ for some $i\in\{1,\dots, d-1\}$, or if $I=\emptyset$, cf. \cite{HRW}). Then, the $k$th intrinsic volume of $K$ can be written as the integral
\begin{equation}\label{Vj}
V_k(K)=\frac 1{\omega_{d-k}}\int_{\nor K} \sum_{|I|=d-1-k}\bK_I(K;x,u)\,\cH^{d-1}(d(x,u)),
\end{equation}
where the sum is taken over all subsets $I$ of $\{1,\ldots,d-1\}$ of given cardinality,
cf.\ \cite[p.~637]{HRW}. Moreover, in \cite{HRW}, (generalized) flag measures  $\Gamma_k(K;\cdot)$ on $\R^d\times F^\perp(d,d-1-k)$ and their image measures $\Omega_k(K;\cdot)=
\Gamma_k(K;\R^d\times \cdot )$ on $F^\perp(d,d-1-k)$ were introduced and it was proved, for measurable functions $g:
\R^d\times F^\perp(d,d-1-k)\to [0,\infty)$, that
\begin{align}\label{geommes}
\int &g(x,u,V)\,\Gamma_k(K;d(x,u,V)) =\frac{\binom{d-1}k}{\omega_{d-k}}\int_{\nor K}\sum_{|I|=d-1-k}\bK_{I}(K;x,u)\\
& \times\int_{G(u^\perp,d-1-k)}g(x,u,V)
\langle V,A_I(K;x,u)\rangle^2\,\nu^{u^\perp}_{d-1-k}(dV)\,\cH^{d-1}(d(x,u)).\nonumber
\end{align}
It follows from \cite[(12) and (25)]{HTW} and \cite[p. 639]{HRW} that the two measures $S^{(d-1-k)}_k(K,\cdot)$ and $\Omega_k(K;\cdot)$ are proportional (they both are averages of the same area measure of projections).
Since 
$$\Gamma_k(K;\R^d\times F^\perp(d,d-1-k) ) = \Omega_k(K;F^\perp(d,d-1-k) ) = V_k(K)$$
(this follows from \cite[p. 641]{HRW}) and 
$$\psi_k(K,F(d,d-k))= \frac{\omega_d}{\omega_{k+1}}S^{(d-1-k)}_k(K,F^\perp (d,d-1-k)) = \frac{\omega_{d-k}}{\binom{d-1}k}V_k(K),$$
we therefore get
\begin{equation}\label{geommes2}
\Omega_k(K;\cdot) = \frac{\binom{d-1}{k}\omega_d}{\omega_{d-k}\omega_{k+1}} S^{(d-1-k)}_k(K,\cdot) ,
\end{equation}
which yields a corresponding measure geometric formula for $\psi_k(K,\cdot)$.

\begin{prop}\label{lemmarelation} We have
\begin{align}\label{psij}
\int g(u,U)\,&\psi_k(K,d(u,U))=\int_{\nor K}\sum_{|I|=k}\bK_{I^c}(K;x,u)\cr 
&\times\int_{G(u^\perp,k)}g(u,V^\perp)\langle V,A_I(K;x,u)\rangle^2\,\nu^{u^\perp}_k(dV)\,\cH^{d-1}(d(x,u)),\quad\quad
\end{align}
for all continuous functions $g$ on $F(d,d-k)$, where $I^c=\{1,\ldots,d-1\}\setminus I$. 
\end{prop}

\begin{proof} From \eqref{conn}, \eqref{geommes} and \eqref{geommes2}, we obtain
\begin{align*}
\int &g(u,U)\,\psi_k(K,d(u,U))\\
&=\frac{\omega_{d-k}}{\binom{d-1}{k}}\int g(u,W+\langle u\rangle)\, \Omega_k(K;d(u,W))\\
&=\int_{\nor K}\sum_{|I|=d-1-k}\bK_{I}(K;x,u)
\int_{G(u^\perp,d-1-k)}g(u,W+\langle u\rangle)\\
&\qquad \times \langle W,A_{I}(K;x,u)\rangle^2\,\nu^{u^\perp}_{d-1-k}(dW)\,\cH^{d-1}(d(x,u))\\
&=\int_{\nor K}\sum_{|I|=d-1-k}\bK_{I}(K;x,u)
\int_{G(u^\perp,d-1-k)}g(u,(W^\perp\cap u^\perp)^\perp)\\
&\qquad \times \langle W^\perp\cap u^\perp,A_{I^c}(K;x,u)\rangle^2\,\nu^{u^\perp}_{d-1-k}(dW)\,\cH^{d-1}(d(x,u))\\
&=\int_{\nor K}\sum_{|I|=d-1-k}\bK_{I}(K;x,u)
\int_{G(u^\perp,k)}g(u,V^\perp)\\
&\qquad \times \langle V,A_{I^c}(K;x,u)\rangle^2\,\nu^{u^\perp}_{k}(dV)\,\cH^{d-1}(d(x,u))\\
&=\int_{\nor K}\sum_{|I|=k}\bK_{I^c}(K;x,u)
\int_{G(u^\perp,k)}g(u,V^\perp)\\
&\qquad \times \langle V,A_I(K;x,u)\rangle^2\,\nu^{u^\perp}_{k}(dV)\,\cH^{d-1}(d(x,u)) ,
\end{align*}
which proves the assertion.
\end{proof}

\section{Proof of Theorem \ref{TH1}}

We start with a slightly different representation of projection functions of polytopes.

\begin{prop}\label{PRP1}
For $1\le k\le d-1$, $P\in {\mathcal P}$, and $E\in G(d,k)$, we have 
\begin{equation}\label{prp1}
v_k(P,E) = \omega_{d-k}^{-1} \sum_{{F\in {\mathcal F}_k(P)}\atop{|\langle E,F\rangle|\not= 0}}\langle E,F\rangle^2V_k(F)\int_{n(P,F)} \|u|E^\perp\|^{k-d}\ {\mathcal H}^{d-1-k}(du)\, .
\end{equation}
\end{prop}

Proposition \ref{PRP1} is a special case of Proposition 4.1 in \cite{Hug}. For completeness, we include a proof.

\begin{proof}
For $v\in E^\perp\setminus \{0\}$, put $E_v^+ := E+v$. Then, for ${\mathcal H}^{d-1-k}$-almost all $v\in E^\perp\cap S^{d-1}$, we have
$$
v_k(P,E) =  \sum_{{F\in {\mathcal F}_k(P)}\atop{N(P,F)\cap E_v^+ \not=\emptyset}} V_k(F|E)\, ,
$$
see formula (39) in \cite[p. 292]{Schn93}. Integration with respect to $v\in S^{d-1-k}(E^\perp)$ yields
\begin{align}\label{3.1}
&\omega_{d-k} v_k(P,E)\cr
&\ = \sum_{F\in {\mathcal F}_k(P)} \int_{S^{d-1-k}(E^\perp)} {\bf 1}\{N(P,F)\cap E_v^+ \not=\emptyset\} |\langle E,F\rangle | V_k(F) \ {\mathcal H}^{d-1-k}(dv)\, . \quad\quad
\end{align}
Now assume $|\langle E,F\rangle |\not= 0$. For fixed $F$ and $E$, let $\pi : \lin N(P,F)\to E^\perp$ be the orthogonal projection onto $E^\perp$ with Jacobian $J\pi(x) = |\langle F^\perp, E^\perp\rangle |=|\langle E,F\rangle |$. Then
\begin{align*}
|\langle E,&F\rangle |\int_{n(P,F)} \|u|E^\perp\|^{k-d}\ {\mathcal H}^{d-1-k}(du)\cr
&= (d-k) \int_{N(P,F)\cap B^d} J\pi(x) \|x/\|x\|\, |\, E^\perp\|^{k-d}\, {\mathcal H}^{d-k}(dx)\cr
&= (d-k) \int_{E^\perp} {\bf 1}\{N(P,F)\cap B^d\cap E_v^+ \not=\emptyset\} \|v^*/\|v^*\|\, | \, E^\perp\|^{k-d}\, {\mathcal H}^{d-k}(dv)\, ,
\end{align*}
where $v^*$ is defined by $\{v^*\} = E_v^+\cap N(P,F)$. Introducing polar coordinates, we obtain
\begin{align}
|\langle E,&F\rangle |\int_{n(P,F)} \|u|E^\perp\|^{k-d}\ {\mathcal H}^{d-1-k}(du)\cr%\nonumber
&= (d-k) \int_{S^{d-1-k}(E^\perp)} \int_0^\infty {\bf 1}\{N(P,F)\cap B^d\cap E_{rv}^+ \not=\emptyset\} \|v^*/\|v^*\|\, | \, E^\perp\|^{k-d}\cr%\nonumber
&\quad \times r^{d-1-k}\, dr\,  {\mathcal H}^{d-1-k}(dv)\cr%\nonumber
&= \int_{S^{d-1-k}(E^\perp)} {\bf 1}\{N(P,F)\cap E_{v}^+ \not=\emptyset\} \,  {\mathcal H}^{d-1-k}(dv)\, .\label{3.2}
\end{align}
Equation \eqref{prp1} now follows from \eqref{3.1} and \eqref{3.2}. \end{proof}

In order to transform \eqref{prp1} into equation \eqref{th1}, we choose $u\in n(P,F)$ and extend it to an orthonormal basis  $u_1=u,u_2,\dots ,u_{d-k}$ of $F^\perp$. If $u\notin E$, then
\begin{align*}
|\langle E,F\rangle| &= |\langle F^\perp,E^\perp\rangle |\cr
&= {\mathcal H}^{d-k} ([0,u_1|E^\perp]+([0,u_2]+\cdots +[0,u_{d-k}])|E^\perp)\cr
&= \| u|E^\perp \| {\mathcal H}^{d-1-k} (([0,u_2]+\cdots +[0,u_{d-k}])|(E^\perp\cap u^\perp))\cr
&= \| u|E^\perp \|\cdot |\langle u^\perp\cap F^\perp, u^\perp\cap E^\perp\rangle | ,
\end{align*}
yielding
\begin{equation}\label{exchange}
\frac{\langle E,F\rangle ^2}{\| u|E^\perp\|^{d-k}} = \frac{\langle u^\perp\cap F^\perp, u^\perp\cap E^\perp\rangle ^2}{\| u|E^\perp\|^{d-k-2}},
\end{equation}
since $\| u|E^\perp\|\neq 0$ by assumption. 

Taking into account the definition \eqref{firstflagmeasure}, Proposition \ref{PRP1} thus yields Theorem \ref{TH1}.

As we mentioned in the introduction, Theorem \ref{TH1} can be extended to arbitrary convex bodies $K$ using approximation by polytopes and a weak compactness argument to define the flag measure $\tau_k(K,\cdot)$ as a weak limit. In view of \eqref{exchange}, the same is true (with the same measure) for the representation in Proposition \ref{PRP1}. More precisely, we get the following result. Here, and subsequently, we use the convention $0\cdot\infty:=0$ for the integrand.

\begin{prop}\label{PRP2}
For $1\le k\le d-1$ and $K\in {\mathcal K}$, there is a Borel measure $\tau_k(K,\cdot)$ on $F(d,d-k)$ such that, for all $E\in G(d,k)$, we have 
\begin{equation}\label{prp2}
v_k(K,E) =\omega_{d-k}^{-1} \int_{F(d,d-k)} \frac{\langle E,U^\perp\rangle ^2}{\|u|E^\perp\|^{d-k}}\ \tau_k(K,d(u,U))\, .
\end{equation}
\end{prop}

\begin{proof}
As we have already explained, we may assume that there is a sequence $(P_i)$ of polytopes converging to $K$, and such that the measures $\tau_k(P_i,\cdot)$ converge weakly, as $i\to \infty$. Let $\tau_k(K,\cdot)$ be the limit measure. 

In the following, we fix $E\in G(d,k)$ and define, for $\varepsilon >0$, the function
$$
f_\varepsilon (u,U) := \omega_{d-k}^{-1}\frac{\langle E,U^\perp\rangle^2}{\max (\varepsilon, \| u|E^\perp\|)^{d-k}}, \quad (u,U)\in F(d,d-k)\, .
$$
This function is continuous on $F(d,d-k)$ and increases, as $\varepsilon$ decreases. The (monotone) limit $f := \lim_{\varepsilon \to 0+}f_\varepsilon $ is finite for all $u\notin E$ and is given by the integrand in \eqref{prp2} (multiplied by $\omega_{d-k}^{-1}$). Combining monotone convergence and weak convergence, we thus get
\begin{align*}
\omega_{d-k}^{-1}\int_{F(d,d-k)}& \frac{\langle E,U^\perp\rangle ^2}{\|u|E^\perp\|^{d-k}}\ \tau_k(K,d(u,U))\cr
&= \lim_{\varepsilon \to 0+} \int_{F(d,d-k)}f_\varepsilon (u,U)\ \tau_k(K,d(u,U))\cr
&= \lim_{\varepsilon \to 0+} \lim_{i\to\infty}\int_{F(d,d-k)}f_\varepsilon (u,U)\ \tau_k(P_i,d(u,U))\cr
&\le \lim_{i\to\infty} \int_{F(d,d-k)}f(u,U)\ \tau_k(P_i,d(u,U))\cr
&=  \lim_{i\to\infty} v_k(P_i,E) = v_k(K,E).
\end{align*}
For the reverse inequality, we use the function $\tilde f_\varepsilon$, given by
$$
\tilde f_\varepsilon (u,U) := f(u,U)\cdot {\bf 1}\{ \| u|E^\perp\|\ge\varepsilon\}, \quad (u,U)\in F(d,d-k)\, ,
$$
which is upper semi-continuous on $F(d,d-k)$. Again by weak convergence, using \cite[Theorem 9.1.5 (v)]{Str},
we obtain
\begin{align*}
\omega_{d-k}^{-1}&\int_{F(d,d-k)} \frac{\langle E,U^\perp\rangle ^2}{\|u|E^\perp\|^{d-k}}\ \tau_k(K,d(u,U))\cr
&\ge  \int_{F(d,d-k)}\tilde f_\varepsilon (u,U)\ \tau_k(K,d(u,U))\cr
&\ge  \limsup_{i\to\infty}\int_{F(d,d-k)}\tilde f_\varepsilon (u,U)\ \tau_k(P_i,d(u,U))\cr
&= \limsup_{i\to\infty}\Bigl( v_k(P_i,E) - \underbrace{\int_{F(d,d-k)}f(u,U){\bf 1}\{ \| u|E^\perp\|<\varepsilon\}\ \tau_k(P_i,d(u,U))}_{=:J_i}\Bigr)\cr
&\ge v_k(K,E) - \varepsilon c(k,K) \cr
&\rightarrow_{\varepsilon\to 0+} v_k(K,E) ,
\end{align*}
provided we have $J_i\le \varepsilon c(k,K)$, for some constant $c(k,K)$ independent of $i$. To show this, we notice that
\begin{align*}
J_i &= \omega_{d-k}^{-1} \sum_{F\in {\mathcal F}_k(P_i)}\langle E,F\rangle^2V_k(F)\int_{n(P_i,F)} \|u|E^\perp\|^{k-d}{\bf 1}\{ \| u|E^\perp\|<\varepsilon\}\, {\mathcal H}^{d-1-k}(du)\cr
&\le  \omega_{d-k}^{-1} \sum_{{F\in {\mathcal F}_k(P_i)}\atop{n(P_i,F) \cap Z(E,\varepsilon)\not=\emptyset}}\langle E,F\rangle^2V_k(F)\int_{n(P_i,F)} \|u|E^\perp\|^{k-d} \, {\mathcal H}^{d-1-k}(du),
\end{align*}
where $Z(E,\varepsilon)\subset S^{d-1}$ is the zonal set
$$
Z(E,\varepsilon) :=\{ u\in S^{d-1} : \| u|E^\perp\| \le\varepsilon \} .
$$
As in the proof of Proposition \ref{PRP1}, we conclude from \eqref{3.2} that
\begin{align*}
J_i &\le  \omega_{d-k}^{-1}\int_{S^{d-1-k}(E^\perp)}\sum_{{F\in {\mathcal F}_k(P_i)}\atop{n(P_i,F)\cap Z(E,\varepsilon)\not=\emptyset}} \cr
&\hspace{3cm}\times V_k(F|E){\bf 1}\{N(P_i,F)\cap E_v^+ \not=\emptyset\}\ {\cal H}^{d-1-k}(dv).\cr
\end{align*}
The argument in \cite[p.~292]{Schn93} shows that the projections $F|E$ of the faces occurring in the sum have pairwise disjoint relative interiors 
for ${\cal H}^{d-1-k}$-almost all $v\in S^{d-1-k}(E^\perp)$. Hence

$$
J_i \le {\cal H}^k\Bigl( \bigcup_{{F\in {\mathcal F}_k(P_i)}\atop{n(P_i,F)\cap Z(E,\varepsilon)\not=\emptyset}}(F|E)\Bigr)
\le \varepsilon c(k,K),
$$
where the final inequality is a consequence of the following lemma.
\end{proof}

\begin{lem}
Let $K\in \mathcal{K}$, $k\in\{1,\ldots,d-1\}$, $L\in G(d,k)$ and $\varepsilon\in(0,1)$. Assume that $x\in \bd K$ has an exterior unit normal
$u$ satisfying $\|u | L\|\ge \sqrt{1-\varepsilon^2}$. Then $x | L\in \relbd (K | L)+\varepsilon D(K) B_L$, where $D(K)$ is the diameter of $K$ and $B_L$ is the unit ball in $L$.
\end{lem}

\begin{proof}
Let $x\in \bd K$ with an exterior unit normal
$u$ such that $\|u | L\|\ge \sqrt{1-\varepsilon^2}$. We may assume that $x=0$. Then $v:=(u|L)/\|u|L\|$ satisfies $\langle u,v\rangle\ge\sqrt{1-\varepsilon^2}$. Choose some $y\in\bd K$ with
exterior unit normal $v$ and put $g:=\lin \{v\}$. Hence, $g\subset L$, $x=x|g=x|L\in K|L$ and $y|g\notin \mathrm{relint}(K|L)$. Thus there is some $z\in[x|g,y|g]\cap \mathrm{relbd}  (K|L)$, and therefore $\|x|L-z\|\le \|x|g-y|g\|$.

Define $E:=\lin \{u,v\}$. Then $x|E$ and $y|E$ are points in the relative
boundary of $K|E$ with exterior unit normals $u$ and $v$, respectively. The convexity of $K|E$ and $\langle u,v\rangle\ge\sqrt{1-\varepsilon^2}$ then imply that
$\|x|g-y|g\|\le \varepsilon \|x|E-y|E\|\le \varepsilon \|x-y\|$, which yields the assertion.
\end{proof}

Proposition \ref{PRP2} shows, in particular, that for each $E\in G(d,k)$ the function $f : (u,U)\mapsto \frac{\langle E,U^\perp\rangle ^2}{\|u|E^\perp\|^{d-k}}$, which may be unbounded, is integrable with respect to $\tau_k(K,\cdot )$.
Using \eqref{exchange} in \eqref{prp2}, we obtain an extension of Theorem \ref{TH1} to arbitrary bodies $K$.

\section{Proof of Theorem~\ref{TH2}}

In this section we present a proof of Theorem~\ref{TH2} (in a slightly more general version) using two substantial ingredients. The first is an integral formula for the projection function (Lemma~\ref{L_Daniel}) proved in  \cite{Hug}; we present the proof here for completeness. The second is an integral formula on the Grassmannian (Lemma~\ref{L_Grassm}),  which we derive by using a technique from \cite{HRW}.
This approach is different from the one presented in \cite{Hind}, but we will use some techniques from \cite{Hind} in the next section to produce an explicit solution of the relevant integral equation.

The following lemma gives the analog of \eqref{Vj}, for projection functions.

\begin{lemma}  \label{L_Daniel}
Given $K\in{\mathcal K}$, $1\leq k\leq d-1$ and $E\in G(d,k)$, we have
\begin{align*}
V_k(K|E)&=\frac 1{\omega_{d-k}}\int_{\nor K}\| u | E^\perp \|^{k-d}\\
&\qquad \times\sum_{|I|=k}\bK_{I^c}(K;x,u)\langle A_I(K;x,u),E\rangle^2\, \cH^{d-1}(d(x,u)).
\end{align*}
\end{lemma}

\begin{proof}
The projection function can be expressed as a mixed volume of $K$ with the unit ball $B^{E^\perp}$ in $E^\perp$ and therefore as a  mixed functional from translative integral geometry,
$$V_k(K|E)=\binom dk\kappa_{d-k}^{-1}V(K[k],B^{E^\perp}[d-k])= \kappa_{d-k}^{-1}V^{(0)}_{k,d-k}(K,B^{E^\perp}),$$
see \cite[(5.68)]{S} and \cite[p. 242]{SW}. 
For the latter, we can use \cite[Eq.~(7)]{HRW}. Notice that if $y$ is an interior point of $B^{E^\perp}$ and $v\in E\cap S^{d-1}$ then $(y,v)\in\nor B^{E^\perp}$ and there are $d-k$ zero principal curvatures at $(y,v)$, say $k_1,\ldots, k_{d-k}$, with principal directions in $E^\perp$, and the remaining $k-1$ principal curvatures are infinite and their principal direction lie in $E$. Hence, if $J_0=\{ d-k+1,\ldots, d\}$ then $\bK_{J_0}(B^{E^\perp};y,v)=1$ and $A_{J_0}(B^{E^\perp};y,v)+\langle v\rangle=E$, and $\bK_{J}(B^{E^\perp};y,v)=0$ for all other index sets $J$ of cardinality $k$. Hence, \cite[Eq.~(7)]{HRW} yields
\begin{align*}
V_k(K|E)&=\frac 1{\kappa_{d-k}}\int_{\nor K}\int_{B^{E^\perp}}\int_{S^{k-1}(E)}F_{k,d-k}(\angle(u,v))\\
&\qquad \times \sum_{|I|=k}\bK_{I^c}(K;x,u)\langle A_I(K;x,u),E\rangle^2\\
&\qquad\qquad \times \cH^{k-1}(dv)\,\cH^{d-k}(dy)\,\cH^{d-1}(d(x,u))\\
&=\int_{\nor K} \sum_{|I|=k}\bK_{I^c}(K;x,u)\langle A_I(K;x,u),E\rangle^2\\
&\qquad\times\int_{S^{k-1}(E)}F_{k,d-k}(\angle(u,v))\,\cH^{k-1}(dv)\,\cH^{d-1}(d(x,u)),
\end{align*}
where 
$$F_{k,d-k}(\theta )=\frac 1{\omega_d}\frac\theta{\sin\theta}\int_0^1\left(\frac{\sin t\theta}{\sin\theta}\right)^{d-1-k} \left(\frac{\sin (1-t)\theta}{\sin\theta}\right)^{k-1}\, dt$$
for $\theta\in [0,\pi)$.
The proof is finished by applying equation (4.1) in \cite{R99} showing that
$$\int_{S^{k-1}(E)}F_{k,d-k}(\angle(u,v))\,\cH^{k-1}(dv)=\frac 1{\omega_{d-k}}\| u | E^\perp\|^{k-d}.$$
\end{proof}

The second ingredient we need is the following integral geometric lemma on the Grassmannian.

\begin{lemma}\label{L_Grassm}
Given $0\leq k\leq d$ and $A\in G(d,k)$, there exists a continuous function $\varphi^d_{k;A}:G(d,k)\to\R$ such that
$$\int_{G(d,k)}\varphi^d_{k;A}(U)\langle U,B\rangle^2\, \nu^d_k(dU)=\langle A,B\rangle^2$$
for any $B\in G(d,k)$.
\end{lemma}

\begin{proof}
In \cite[Lemma~4]{HRW}, it is shown that there exist real constants $d_{p,q}^{d,k}$, $0\leq p,q\leq k\wedge(d-k)$, such that
for any $A,B\in G(d,k)$ and $0\le p\le k\wedge (d-k)$,
\begin{equation} \label{L4}
\int_{G(d,k)}\langle A,U\rangle_p^2\langle U,B\rangle^2\, \nu^d_k(dU)=\sum_{q=0}^{k\wedge(d-k)}d_{p,q}^{d,k}\langle A,B\rangle_q^2.
\end{equation}
Set
$$\varphi^d_{k;A}(U):=\sum_{p=0}^{k\wedge(d-k)}\alpha_p\langle A,U\rangle_p^2,\quad U\in G(d,k),$$
for some constants $\alpha_p$ which will be specified below. Using \eqref{L4}, we obtain
$$\int_{G(d,k)}\varphi^d_{k;A}(U)\langle U,B\rangle^2\, \nu^d_k(dU)=\sum_{p=0}^{k\wedge(d-k)}\sum_{q=0}^{k\wedge(d-k)}\alpha_pd_{p,q}^{d,k}\langle A,B\rangle_q^2.$$
Since the matrix $D=(d_{p,q}^{d,k})_{p,q=0}^{k\wedge(d-k)}$ is regular (see \cite[Proposition~1]{HRW}), we can choose
$$(\alpha_0,\ldots,\alpha_{k\wedge(d-k)})=(1,0,\ldots,0)D^{-1}$$
and obtain
$$\int_{G(d,k)}\varphi^d_{k;A}(U)\langle U,B\rangle^2\, \nu^d_k(dU)=\langle A,B\rangle_0^2=\langle A,B\rangle^2,$$
completing the proof.
\end{proof}

We can now proceed with the proof of Theorem~\ref{TH2}. The idea is to apply the integral representation  \eqref{psij} for the flag measure $\psi_k(K,\cdot)$ with a function $g= g(E,\cdot)$ arising from Lemma \ref{L_Grassm}. To be more precise, let $E\in G(d,k)$ and $u\in S^{d-1}\setminus E$. Then $E^\perp\cap u^\perp\in G(u^\perp,d-1-k)$ and we can apply Lemma \ref{L_Grassm} in $u^\perp$ to obtain a function $\varphi^{u^\perp}_{d-1-k; E^\perp\cap u^\perp}$ such that
$$
\int_{G(u^\perp,d-1-k)}\varphi^{u^\perp}_{d-1-k; E^\perp\cap u^\perp}(W)\langle W,B\rangle^2\,
\nu^{u^\perp}_{d-1-k}(dW)=\langle E^\perp\cap u^\perp,B\rangle^2
$$
for all $B\in G(u^\perp,d-1-k)$. In particular, for $B=A_{I^c}$ with $|I|=k$
(we omit the argument and identify the Grassmannian and the oriented Grassmannian), we get
\begin{align}\label{identify}
&\int_{G(u^\perp,k)}\varphi^{u^\perp}_{d-1-k; E^\perp\cap u^\perp}(W^\perp\cap u^\perp)\langle W ,A_{I}\rangle^2\,
\nu^{u^\perp}_{k}(dW) \nonumber\\
&\qquad=\int_{G(u^\perp,k)}\varphi^{u^\perp}_{d-1-k; E^\perp\cap u^\perp}(W^\perp\cap u^\perp)\langle W^\perp\cap u^\perp,A_{I^c}\rangle^2\,
\nu^{u^\perp}_{k}(dW) \nonumber\\
&\qquad= \int_{G(u^\perp,d-1-k)}\varphi^{u^\perp}_{d-1-k; E^\perp\cap u^\perp}( U)\langle U,A_{I^c}\rangle^2\,\nu^{u^\perp}_{d-1-k}(dU) \nonumber\\
&\qquad=\langle E^\perp\cap u^\perp,A_{I^c}\rangle^2.
\end{align}
For $(u,U)\in F(d,d-k)$ we define
\begin{equation}\label{gdef}
g(E,u,U):=\begin{cases}
\omega_{d-k}^{-1}\|u|E^\perp\|^{k-d+2}\varphi^{u^\perp}_{d-1-k; E^\perp\cap u^\perp}(U\cap u^\perp ),&u\notin E,\\
0,&u\in E.
\end{cases}
\end{equation}
Then we have, using  equations \eqref{psij}, \eqref{gdef}, \eqref{identify}, \eqref{exchange} and Lemma \ref{L_Daniel},
\begin{align*}
&\int g(E,u,U)\, \psi_k(K,d(u,U))\\
&\qquad = \int_{\nor K}\sum_{|I|=k}\bK_{I^c}(K;x,u)
\int_{G(u^\perp,k)}g(E,u,W^\perp)\\
\displaybreak
&\qquad\qquad \times \langle W,A_{I}(K;x,u)\rangle^2\,\nu^{u^\perp}_{k}(dW)\,\cH^{d-1}(d(x,u))\\
\pagebreak
&\qquad = \frac{1}{\omega_{d-k}}\int_{\nor K}\sum_{|I|=k}\bK_{I^c}(K;x,u)\mathbf{1}\{u\notin E\}
\\
&\qquad\qquad \times \int_{G(u^\perp,k)}\|u|E^\perp\|^{k-d+2}
\varphi^{u^\perp}_{d-1-k; E^\perp\cap u^\perp}( W^\perp\cap u^\perp)\\
&\qquad\qquad \times\langle W,A_{I}(K;x,u)\rangle^2\,\nu^{u^\perp}_{k}(dW)\,\cH^{d-1}(d(x,u))\\
&\qquad = \frac{1}{\omega_{d-k}}\int_{\nor K}\sum_{|I|=k}\bK_{I^c}(K;x,u)\mathbf{1}\{u\notin E\}
\|u|E^\perp\|^{k-d+2}\\
&\qquad\qquad\times \langle E^\perp\cap u^\perp,A_{I^c}(K;x,u)\rangle^2\,\cH^{d-1}(d(x,u))\\
&\qquad = \frac{1}{\omega_{d-k}}\int_{\nor K}\sum_{|I|=k}\bK_{I^c}(K;x,u)\mathbf{1}\{u\notin E\}
\|u|E^\perp\|^{k-d}\\
&\qquad\qquad\times\langle E,A_{I}(K;x,u)\rangle^2\,\cH^{d-1}(d(x,u))\\
&\qquad =V_k(K|E),
\end{align*}
where we applied the identity
$$\langle E,A_I\rangle^2=\langle E^\perp\cap u^\perp,A_{I^c}\rangle^2 \| u | E^\perp \|^2$$
in the second to last step and Lemma~\ref{L_Daniel} in the final step. Note that we can omit the indicator function $\mathbf{1}\{u\notin E\}$,
since $u\in E$ implies that $\langle E,A_I\rangle^2=0$ and $0\cdot\infty=0$. 
Thus we obtain \eqref{th2} provided the integrals exist.  

The problem is that the continuous function $\varphi^d_{k;A}$ in Lemma \ref{L_Grassm} will have negative parts, in general (see Theorem \ref{explsol}, for an explicit formula for the coefficients $\alpha_p$). Since the projection factor in $g$ is not bounded, for $k<d-2$, the integrals in the above calculation need not exist (an example of this kind is constructed in \cite{HRW}, in the related case of mixed volumes). The following theorem collects the cases where we have a positive result. Theorem \ref{TH2} follows from part (b).

\begin{theorem} \label{TH2detailed} The assertion of Theorem \ref{TH2} holds for all $K$ and $E$, if $k\in \{d-1,d-2\}$. For the other values of $k$ and given $K\in\mathcal K$, formula \eqref{th2} is valid

\smallskip\noindent
{\bf (a)} for $\rho K$ and $E\in G(d,k)$, for $\nu$-almost all $\rho\in O(d)$,

\smallskip\noindent
{\bf (b)} for $K$ and $\nu^d_k$-almost all $E$,

\smallskip\noindent
{\bf (c)} for $K$ and all $E$ if $S_k(K,\cdot) \ll {\mathcal H}^{d-1}$ with bounded density,

\smallskip\noindent
{\bf (d)} for $K\in{\mathcal P}$ and $E\in G(d,k)$, if $K$ and $E$ are in general relative position, i.e. $L(F)\cap E^\perp =\{ 0\}$, for all $F\in {\mathcal F}_k(K)$.
\end{theorem}

\begin{proof} The cases $k=d-1$ and $k=d-2$ are covered by the argument given before the theorem. Hence, we now assume $k< d-2$.
 
For $\ee >0$, we define an $\ee$-approximation $v_k^{\ee}(K,E)$ of $v_k(K,E)$ by
%\pagebreak
\begin{align*}
v_k^{\ee}(K,E) &:=\frac 1{\omega_{d-k}}\int_{\nor K}{\bf 1}\{\| u | E^\perp \|\ge\ee\}\| u | E^\perp \|^{k-d}\\
&\qquad\times\sum_{|I|=k}\bK_{I^c}(K;x,u)\langle A_I(K;x,u),E\rangle^2\, \cH^{k-1}(d(x,u)).
\end{align*}
Then, monotone convergence implies that
$$
v_k^{\ee}(K,E)\nearrow v_k(K,E) , \quad {\rm as}\ \ee\to 0 .
$$
Hence we have to show that
\begin{align}\label{limit}
\int_{F(d,d-k)} &{\bf 1}\{\| u | E^\perp \|\ge\ee\} g(E,u,U)\,\psi_k(K,d(u,U)) \\
&\to \int_{F(d,d-k)}   g(E,u,U)\,\psi_k(K,d(u,U)) \quad {\rm as\ } \ee\to 0\nonumber
\end{align}
in each of the cases (a)--(d), with $g(E,u,U)$ given by \eqref{gdef}. Obviously, the function $|\varphi^{d-1}_{k;E'}|$ can be bounded by a constant $c>0$.

\smallskip\noindent
(a) In order to apply the dominated convergence theorem in \eqref{limit} we show that
\begin{align*}
\int_{F(d,d-k)}& \| u | E^\perp \|^{k-d+2}\,\psi_k(\rho K,d(u,U))\\
&= \int_{S^{d-1}} \| u | E^\perp \|^{k-d+2}\,S_k(\rho K,du) < \infty ,
\end{align*}
for $\nu$-almost all $\rho\in O(d)$. This will follow from
$$
\int_{O(d)}\int_{S^{d-1}} \| u | E^\perp \|^{k-d+2}\,S_k(\rho K,du)\, \nu(d\rho)< \infty.
$$
To see the latter, note that
\begin{align*}
\int_{O(d)}&\int_{S^{d-1}} \| u | E^\perp \|^{k-d+2}\,S_k(\rho K,du)\, \nu(d\rho)\\
&= \int_{O(d)}\int_{S^{d-1}} \| \rho u | E^\perp \|^{k-d+2}\,S_k(K,du)\, \nu(d\rho)\\
&= \frac{\omega_{d-k}}{\binom{d-1}{k}}V_k(K) \int_{S^{d-1}} \| v| E^\perp \|^{k-d+2}\,{\mathcal H}^{d-1}(dv)\\
&= \frac{\omega_{d-k}}{\binom{d-1}{k}}V_k(K) \int_{S^{k-1}(E)} \int_{S^{d-1-k}(E^\perp)}\int_0^{\pi/2}\| (\cos \theta \ u_1,\sin \theta \  u_2)| E^\perp \|^{k-d+2}\cr
&\quad\quad\times(\cos \theta)^{k-1}(\sin \theta )^{d-1-k}\,d\theta\,{\mathcal H}^{d-1-k}(du_2)\,{\mathcal H}^{k-1}(du_1)\\
%\displaybreak
&= \frac{\omega_{d-k}}{\binom{d-1}{k}}V_k(K)\omega_{k}\omega_{d-k}\int_0^{\pi/2} (\sin \theta )^{k-d+2}(\cos \theta)^{k-1}(\sin \theta )^{d-1-k}\,d\theta\\ 
&= \frac{\omega_{d-k}}{\binom{d-1}{k}}V_k(K)\omega_{k}\omega_{d-k}\int_0^{\pi/2} (\cos \theta)^{k-1}\sin \theta\, d\theta  \\
&< \infty .
\end{align*}

\smallskip\noindent
(b) From (a) we obtain, for $\nu$-almost all $\rho\in O(d)$,
\begin{align*}
v_k(K,\rho E) &= v_k(\rho^{-1}K,E)\\
&= \int_{F(d,d-k)} \| u | E^\perp \|^{k-d+2}\varphi^{d-1}_{k,E\cap u^\perp}(U)\,\psi_k(\rho^{-1} K,d(u,U))\\
&= \int_{F(d,d-k)} \| \rho^{-1}u | E^\perp \|^{k-d+2}\varphi^{d-1}_{k,E\cap (\rho^{-1}u)^\perp}(\rho^{-1}U)\,\psi_k(K,d(u,U))\\
&= \int_{F(d,d-k)} \| u | (\rho E)^\perp \|^{k-d+2}\varphi^{d-1}_{k,(\rho E)\cap u^\perp}(U)\,\psi_k(K,d(u,U)) ,
\end{align*}
where we used the covariance properties of $\psi_k$ and $\varphi^{d-1}_{k,\cdot}(\cdot)$. Replacing $\rho E$ by $E'$, we obtain the assertion for $\nu^d_k$-almost all $E'\in G(d,k)$.

\smallskip\noindent
(c) follows from the argument used in the proof of (a).

\smallskip\noindent
(d) For a polytope $K$ we have
$$
S_k(K,\cdot) = \binom{d-1}k^{-1} \sum_{F\in{\mathcal F}_k(K)} V_k(F){\mathcal H}^{d-1-k}(n(K,F)\cap \cdot) .
$$
Hence the support of $S_k(K,\cdot)$ is contained in 
$$\bigcup_{F\in{\mathcal F}_k(K)} n(K,F).
$$
Since $L(F)\cap E^\perp = \{ 0\}$ for all $F\in{\mathcal F}_k(K)$, we get $L(F)^\perp\cap E = \{ 0\}$ and therefore $n(K,F)\cap E=\emptyset $. A compactness argument then shows that there is some $\ee > 0$ such that $\| u | E^\perp\| \ge \ee$ for all $u$ in the support of $S_k(K,\cdot)$.
\end{proof}

\section{Determination of a concrete solution}

The proof of Theorem \ref{TH2} given in the previous section shows that a continuous function $g$ exists which yields the integral representation \eqref{th2}. In the following, we give an explicit expression for this function. This will follow from the proof of Lemma~\ref{L_Grassm} if we have an explicit solution ${\alpha} :=(\alpha_0 ,\dots ,\alpha_{k\wedge (d-k)})$ of the linear system
\begin{equation}\label{LSE}
{\alpha}\cdot D = (1,0,\dots ,0)
\end{equation}
with coefficient matrix $D=(d^{d,k}_{m,i})_{m,i=0}^{k\wedge (d-k)}$. The following theorem gives such a solution.

\begin{theorem}\label{explsol}
An explicit solution $\alpha$ of \eqref{LSE} is given by
$$
\alpha_m = (-1)^{m}  \frac{\binom{d}{k}\binom{d+1}{k}}{(k+1)\binom{d+1}{m}}, \ m=0,\dots , k\wedge (d-k) .
$$
\end{theorem}

In order to prepare the proof of Theorem \ref{explsol}, we recall from \cite[p. 649]{HRW} the equation
\begin{align}\label{abbr}
d^{d,k}_{m,i}
&= \sum_{j=0}^{k\wedge (d-k)} c^d_{k,j} F(d,k,j,m,i)
\end{align}
where 
$$
 F(d,k,j,m,i):=
\sum_{l=0}^k \binom{k-i}{l}\binom{i}{k-m-l}\binom{i}{k-j-l}\binom{d-k-i}{j+l+m-k}
$$
and where the constants $c^d_{k,j}$, for $0\le j\le k\wedge (d-k)$, are defined as follows. Let $e_1,\dots ,e_d$ be the canonical basis of $\R^d$ and let $E\in G(d,k)$ be the subspace spanned by $e_1,\dots , e_k$ and $F\in G(d,k)$ the subspace spanned by $e_1,\dots ,e_{k-j},$ $e_{k+1},\dots ,e_{k+j}$. Then 
$$
c^d_{k,j} = \int_{G(d,k)} \langle E,V\rangle^2\langle F,V\rangle ^2 \nu^d_k(dV)
$$
depends only on the dimensions $d,k,j$ (and not on the particular choice of the bases), as follows from \cite[Lemma 3]{HRW} and the explanation on p. 647 of \cite{HRW}.

In the next two lemmas it is convenient to use calculations in the Grassmann algebra $\bigwedge \R^d$ of $\R^d$. As in \cite{HRW}, we introduce scalar product (and a norm)  
on $\bigwedge_k\R^d$,  
which are induced by the scalar product of $\R^d$, and we also use the same notation.

\begin{lemma}   \label{L1c}
For any nonnegative integers $i,k,d$ such that $0\leq i<k$ and $2k\leq d$ we have
$$c^d_{k,i}=I_k^d\,c^{d-1}_{k-1,i},$$
where
$$I^d_k=\frac{k(k+2)}{d(d+2)}.$$
\end{lemma}

\begin{proof}
Let $E,F$ be defined as above and let $V$ be another $k$-subspace. If $V$ is not orthogonal to $e_1$, then $V$ has an orthonormal basis $v_1,\ldots,v_k$ such that $v_1,\ldots,v_{k-1}$ are orthogonal to $e:=e_1$ and $v_k$ is the normalized orthogonal projection of $e_1$ to $V$. Then we have
\begin{align*}
\langle E,V\rangle^2&=\|e_{k+1}\wedge\dots\wedge e_d\wedge v_1\wedge\dots\wedge v_k\|^2\\
&=|\langle v_k,e\rangle|^2 \|e_{k+1}\wedge\dots\wedge e_d\wedge v_1\wedge\dots\wedge v_{k-1}\|^2\\
&=\| e | V\|^2 \langle E\cap e^\perp,V\cap e^\perp\rangle^2
\end{align*}
and a similar relation for $F$.
Since
$$c^d_{k,i}=\int_{G(d,k)} \langle E,V\rangle^2\langle F,V\rangle^2\, \nu_k^d(dV),$$
the above identity yields
$$c^d_{k,i}=\int_{G(d,k)}  \| e | V\|^4\langle E\cap e^\perp,V\cap e^\perp\rangle^2\langle F\cap e^\perp,V\cap e^\perp\rangle^2\, \nu_k^d(dV).$$
We now apply a special case of Lemma~4.1 in \cite{HSS08} (see also \cite[Theorem 7.2.6]{SW}) which yields
\begin{align}  \label{HSS}
\int_{G(d,k)}h(V)\,\nu_k^d(dV) &= 
\frac{\Gamma\left(\frac d2\right)}{2\pi^{\frac{d-k}2}\Gamma\left(\frac k2\right)}
\int_{G(e^\perp,k-1)}\\
&\times\int_{S^{d-k}(W^\perp)}h(W+ \langle v\rangle)\, |\langle e,v\rangle|^{k-1}
\cH^{d-k}(dv)\,\nu^{e^\perp}_{k-1}(dW),  \nonumber
\end{align}
whenever $h$ is an integrable function on $G(d,k)$. For our choice of $h$ we obtain
\begin{align*}
c^d_{k,i}&=\frac{\Gamma\left(\frac d2\right)}{2\pi^{\frac{d-k}2}\Gamma\left(\frac k2\right)}
\int_{G(e^\perp,k-1)}\int_{S^{d-k}(W^\perp)}|\langle v,e\rangle|^4 |\langle v,e\rangle|^{k-1}\, \cH^{d-k}(dv)\\
 &\hspace {1cm}\times\langle E\cap e^\perp,W\rangle^2\langle F\cap e^\perp,W\rangle^2\, \nu^{e^\perp}_{k-1}(dW)\\
&= \frac{\Gamma\left(\frac d2\right)}{2\pi^{\frac{d-k}2}\Gamma\left(\frac k2\right)}
\int_{S^{d-k}(W^\perp)}|\langle v,e\rangle|^{k+3}\, \cH^{d-k}(dv)\, c^{d-1}_{k-1,i}\\
&= \frac{\Gamma\left(\frac d2\right)}{2\pi^{\frac{d-k}2}\Gamma\left(\frac k2\right)}
\frac{2\pi^{\frac{d-k}2}\Gamma\left(\frac{k+4}2\right)}{\Gamma\left(\frac{d+4}2\right)} c^{d-1}_{k-1,i}\\
&=\frac{k(k+2)}{d(d+2)}c^{d-1}_{k-1,i},
\end{align*}
yielding the assertion. Here, we have used the first equality
from Lemma~\ref{L_int_sphere} in the last but one step. 
\end{proof}

\begin{lemma}  \label{L2c}
For any positive integers $k,d$ such that $2k\leq d$ we have
$$c^d_{k,k}=J_k^d\,c^{d-1}_{k-1,k-1},$$
where
$$J^d_k=\frac{k^2}{d(d+2)}.$$
\end{lemma}

\begin{proof}
Let $E,F$ be the subspaces as defined above. Let $V$ be another $k$-subspace which is not orthogonal to $e=e_1$, let $v_1,\ldots,v_k$ be the orthogonal basis of $V$ as in the proof of Lemma~\ref{L1c}, and let $u$ be a unit vector from $F\cap (V\cap e^\perp)^\perp$. We set
$$
w:=e_1\wedge\dots\wedge e_k\wedge e_{2k+1}\wedge\dots\wedge e_d\wedge v_1\wedge\dots\wedge v_{k-1}
$$
and choose a unit vector $a$ orthogonal to the vectors involved in the definition of $w$. Then
\begin{align*}
\sum_{j=k+1}^{2k}\|w\wedge e_j\|^2=\sum_{j=1}^{d}\|w\wedge e_j\|^2=\sum_{j=1}^{d}\langle e_j,a\rangle^2 \|w\wedge a\|^2
=\|w\wedge a\|^2=\|w\|^2.
\end{align*}
Since $\langle F\cap e_j^\perp,V\cap e^\perp\rangle^2=\|w\wedge e_j\|^2$, we obain
\begin{align*}
\langle F,V\rangle^2&=\|e_1\wedge\dots\wedge e_k\wedge e_{2k+1}\wedge\dots\wedge e_d\wedge v_1\wedge\dots\wedge v_k\|^2=|\langle u,v_k\rangle|^2 \|w\|^2\\
&=|\langle u,v_k\rangle|^2 \sum_{j=k+1}^{2k}\langle F\cap e_j^\perp,V\cap e^\perp\rangle^2.
\end{align*}
Hence,
$$c^d_{k,k}=\int_{G(d,k)} |\langle e,v_k\rangle|^2|\langle u,v_k\rangle|^2\sum_{j=k+1}^{2k}\langle E\cap e^\perp,V\cap e^\perp\rangle^2\langle F\cap e_j^\perp,V\cap e^\perp\rangle^2\, \nu_k^d(dV).$$
Using \eqref{HSS} again, we get
\begin{align*}
c^d_{k,k}&=\frac{\Gamma\left(\frac d2\right)}{2\pi^{\frac{d-k}2}\Gamma\left(\frac k2\right)}\int_{G(e^\perp,k-1)}\int_{S^{d-k}(W^\perp)}|\langle e,v\rangle|^2 |\langle u,v\rangle|^2 |\langle e,v\rangle|^{k-1}\, \cH^{d-k}(dv)\\
 &\hspace {1cm}\times\sum_{j=k+1}^{2k}\langle E\cap e^\perp,W\rangle^2\langle F\cap e_j^\perp,W^\perp\rangle^2\, \nu^{e^\perp}_{k-1}(dW)\\
&= \frac{\Gamma\left(\frac d2\right)}{2\pi^{\frac{d-k}2}\Gamma\left(\frac k2\right)}k\int_{S^{d-k}(W^\perp)}|\langle e,v\rangle|^{k+1} |\langle u,v\rangle|^2 \, \cH^{d-k}(dv)\,  c^{d-1}_{k-1,k-1}.
\end{align*}
The second equality from Lemma~\ref{L_int_sphere} yields
\begin{align*}
\int_{S^{d-k}}|\langle e,v\rangle|^{k+1} |\langle u,v\rangle|^2 \, \cH^{d-k}(dv)&=2\pi^{\frac{d-1-k}2}
\frac{\Gamma\left(\frac{k+2}2\right)\Gamma\left(\frac 32\right)}
{\Gamma\left(\frac{d+4}2\right)} \\ &=
\pi^{\frac{d-k}2}\frac{\Gamma\left(\frac{k+2}2\right)}{\Gamma\left(\frac{d+4}2\right)},
\end{align*}
hence we have
$$c^d_{k,k}=\frac{k^2}{d(d+2)}c^{d-1}_{k-1,k-1}.$$
\end{proof}

\begin{koro}\label{cor1}
For integers $0\le k\le d-1$ and $0\le i\le k\wedge (d-k)$, we have
\begin{equation}\label{expl.expr2}
c^d_{k,i}=\frac{(k+1)(k+2)}{(i+1)(i+2)}{\binom dk}^{-1} {\binom{d+2}k}^{-1}.
\end{equation}
\end{koro}

\begin{proof}
For $2k\le d$, we can iterate the recursive formulas from Lemma~\ref{L1c} and Lemma~\ref{L2c} and get
\begin{equation*}
c^d_{k,i}=I^d_kI^{d-1}_{k-1}\cdots I^{d-k+i+1}_{i+1}J^{d-k+i}_{i}J^{d-k+i-1}_{i-1}\cdots J^{d-k+1}_1 ,
\end{equation*} which yields \eqref{expl.expr2}. 

For $2k>d$, we observe that the right-hand side of \eqref{expl.expr2} does not change if we replace $k$ by $d-k$. Hence, we only need to show that $c^d_{k,i}= c^d_{d-k,i}$ holds (for $2k>d$ and $0\le i\le (d-k)$). Let $E,F\in G(d,k)$ have a $(k-i)$-dimensional intersection. Then, 
\begin{align*}
c^d_{k,i} &=  \int_{G(d,k)} \langle E,V\rangle^2\langle F,V\rangle ^2 \nu^d_k(dV)\cr
&= \int_{G(d,d-k)} \langle E,V^\perp\rangle^2\langle F,V^\perp\rangle ^2 \nu^d_{d-k}(dV)\cr
&=\int_{G(d,d-k)} \langle E^\perp,V\rangle^2\langle F^\perp,V\rangle ^2 \nu^d_{d-k}(dV)\cr
&= c^d_{d-k,i} ,
\end{align*}
since $E^\perp \cap F^\perp$ has dimension $d-k-i$.
\end{proof}

Using \eqref{abbr}, \eqref{expl.expr2} and the formula for $\alpha_m$, we obtain
\begin{align*}
\sum_m&\alpha_md^{d,k}_{m,i}\\
&=\sum_m\sum_j (-1)^m \binom{d+1}{m}^{-1}\frac{(d+2-k)(k+2)}{(d+2)(j+1)(j+2)} F(d,k,m,j,i)
\\
&=\frac{(d+2-k)(k+2)}{2(d+2)}\sum_m(-1)^m \binom{d+1}{m}^{-1}\sum_j\binom{j+2}{2}^{-1}F(d,k,m,j,i).
\end{align*}
We replace $m$ by $k-m$ and $j$ by $k-j$. Then, we get from Theorem \ref{identity}
\begin{align*}
\sum_m\alpha_md^{d,k}_{m,i}&=\frac{(d+2-k)(k+2)}{2(d+2)}\sum_m (-1)^{k-m} \binom{d+1}{k-m}^{-1}\\
&\qquad\times\sum_j\binom{k-j+2}{2}^{-1}F(d,k,k-m,k-j,i)\\
&=\binom{k-i}{k}
\end{align*}
which is $1$ for $i=0$ and $0$ otherwise. This completes the proof of Theorem \ref{explsol}.

\section{Final discussion}

We add a few remarks concerning the functional analytic aspects of the two integral representations which we obtained for the projection functions. 
Let $C(F(d,d-k))$ denote the Banach space of continuous functions on $F(d,d-k)$. We consider the integral transform
$T_{k} : C(F(d,d-k))\to C(F(d,d-k))$, given by
\begin{equation*}
(T_{k}h)(u,L) 
:= \int_{G(\langle u\rangle ,d-k)} \langle L,M\rangle ^2 h(u,M) \nu^{\langle u\rangle}_{d-k}(dM) ,
\end{equation*}
for $(u,L)\in F(d,d-k)$. The transform $T_k$ is obviously continuous and self-adjoint, therefore it can be extended to a weakly continuous linear map on the dual space of finite signed measures on $F(d,d-k)$. 
It follows from the definitions \eqref{firstflagmeasure} and \eqref{secondflagmeasure} that, for a polytope $P$, the measure $\psi_k(P,\cdot )$ is the image of $\tau_k(P,\cdot )$ under $T_{k}$. This also holds for convex bodies $K\in{\mathcal K}$, if we choose $\tau_k(K,\cdot )$ as the limit of a converging sequence  $\tau_k(P_i,\cdot )$, as we did in Section 3. The non-uniqueness of this limit implies that the linear map $T_k$ is not injective, as was already remarked in \cite{HHW}. 

Using this mapping, we could deduce Theorem \ref{TH2}  for polytopes $P$ from 
Theorem \ref{TH1} (up to a constant), if we find a function $g$ satisfying the integral equation
\begin{equation*}
\int_{G(\langle u\rangle ,d-k)} \langle U,F^\bot\rangle ^2 g(E,u,U) \nu^{\langle u\rangle}_{d-k}(dU) = \frac{\langle u^\bot\cap F^\bot, u^\bot\cap E^\bot\rangle^2}{\| u | E^\bot \|^{d-k-2}} 
\end{equation*}
for $E,F\in G(d,k)$ and $u\bot F$. 
For a discussion of the latter, we may fix $u\in S^{d-1}$ and  replace $u^\perp$ by $\R^n$ ($n=d-1$). Instead of integrating over all $U\in G(\langle u\rangle,d-k)$, we may then integrate over all $L\in G(n,n-k)$ and seek a function $f_E$ on $G(n,n-k)$ such that 
\begin{equation}\label{inteq}
\int_{G(n,n-k)} \langle L,F\rangle^2 f_E(L) \nu_{n-k}(dL) = \langle E,F\rangle^2 
\end{equation}
holds for all $E,F\in G(n,n-k)$. As we have seen in Lemma \ref{L_Grassm},  \eqref{inteq} has a solution, for each $E\in G(n,n-k)$, given by a continuous function  $f_E$ on $G(n,n-k)$. Since $E$ here is an arbitrary subspace, we may choose $E$ to be the linear hull of the vectors $e_1,\dots ,e_{n-k}$ in the standard basis $e_1,\dots ,e_n$ of $\R^n$ and consider the integral operator $S_{k} : C(G(n,n-k))\to C(G(n,n-k))$, defined by
$$
S_{k}(h)(F) := \int_{G(n,n-k)} \langle L,F\rangle^2 h(L) \nu_{n-k}(dL), $$
for $F\in G(n,n-k), h\in C(G(n,n-k))$. The mapping 
$S_{k}$ is a slight variant of the mapping $T_k$ mentioned above. Again, $S_{k}$ is self-adjoint, but not injective. However, if we consider the subspace ${\mathcal L}\subset C(G(n,n-k))$ considered in \cite{HRW}, then $S_{k}$ maps ${\mathcal L}$ to ${\mathcal L}$ and, when restricted to ${\mathcal L}$, it is a bijection by \cite[Corollary 1]{HRW}. To recall the definition of ${\mathcal L}$, we define $E_I :=\lin\{ e_i : i\in I\}$, where $I\subset \{1,\dots ,n\}$. Then, ${\mathcal L}$ is the $n\choose {n-k}$-dimensional linear space of continuous functions on $G(n,n-k)$ spanned by the functions $\langle E_I,\cdot\rangle^2$ with $|I|=n-k$. Since for the above choice of $E$, we trivially have $\langle E,\cdot\rangle^2 \in {\mathcal L}$, the function $f_E$ can be defined as the inverse image of $\langle E,\cdot\rangle^2$ in ${\mathcal L}$. Actually, this is the approach we used in the proof of Theorem \ref{TH2}. The operator $S_k$ is the special case $\alpha =2$ of the $\alpha$-cosine transform, which has gained recent attention (see, for example, \cite{AGS} and the literature cited there).

We also add some comments on the flag formula \eqref{th2} for the projection function $v_k(K,\cdot)$. First, the special case $k=d-1$ yields \eqref{projfunc} again. Namely, \eqref{conn} shows that $\psi_{d-1}(K,\cdot) = S_{d-1}^{(0)}(K,\cdot)$ and the latter equals $S_{d-1}(K,\cdot)$ under the natural identification $u\leftrightarrow (u,\{ 0\})$. Then, the function $g(E,u,U)$, with the projection hyperplane $E$ and $U=\langle u\rangle$, must equal $\frac{1}{2} |\langle u, v\rangle|$ (up to a linear function), where $v$ is the normal of $E$.  The corresponding question as to whether \eqref{th2} implies \eqref{projfunc2} if $K$ is a generalized zonoid seems to be more complicated. Although there are formulas expressing the flag measure $\psi_k(K,\cdot)$ of a generalized zonoid $K$ in terms of the generating measure $\rho (K,\cdot)$ of $K$ (see \cite[Corollary 3]{HTW}), a possible  connection between  \eqref{th2} and \eqref{projfunc2} is complicated by the fact that a signed measure $\phi$ on $G(d,k)$ satisfying
$$
v_{k}(K,E) =  \frac{2^k}{k!} \int_{G(d,k)} |\langle E,F\rangle | \phi(dF) , \quad E\in G(d,k) ,
$$
is not uniquely determined, if $k\in\{ 2,\dots ,d-2\}$ (see \cite{GH}).

Using \eqref{conn}, we can formulate the projection formula \eqref{th2} also with $\psi_k(K,\cdot)$ replaced by $S_k^{(d-k-1)}(K,\cdot)$. Namely, let $\tilde g$ be defined by
$$
\tilde g(E,u,V) := \frac{\omega_d}{\omega_{k+1}}g(E,u,V+\langle u\rangle),\quad E\in G(d,k), (u,V)\in F^\perp(d,d-1-k).
$$
Then we have
\begin{equation}\label{projvariant}
v_k(K,E) = \int_{F^\perp(d,d-1-k)} \tilde g(E,u,V) S_k^{(d-k-1)}(K,d(u,V))
\end{equation}
for $K\in{\mathcal K}$ and almost all $E\in G(d,k)$. Formula \eqref{projvariant} can easily be generalized to lower order projection functions
$$
v_{ik}(K,E) := V_i(K|E),\quad E\in G(d,k),
$$
$i=0,\dots, k$, and to mixed projection functions
$$
v(K_1,\dots ,K_k,E) := V^{(E)}(K_1|E,\dots ,K_k|E), \quad E\in G(d,k),
$$
(where $V^{(E)}(M_1,\dots M_k)$ denotes the mixed volume (in $E$) of convex bodies $M_1,\dots ,M_k\subset E$), since both sides of \eqref{projvariant} allow a polynomial expansion for $K+rB^d, r\ge 0$, respectively a multinomial expansion for $r_1K_1+\dots +r_m K_m, K_i\in{\mathcal K}, r_i\ge 0$. Here, for the right-hand side, Proposition 2 and Theorem 6 in \cite{HTW} can be used, the latter introducing mixed flag measures. Explicit formulas are given in \cite{Hind} and \cite{Zieb}.

Since projection functions in general do not determine the shape of a non-symmetric convex body, one may also ask for flag representations of directed projection functions, as they were introduced and studied in \cite{GW06a,GW06b}. We leave this as a question for further investigations.

\section{Appendix}

The following lemma was used in Section 5.

\begin{lemma} \label{L_int_sphere}
If $d\geq 2$, $u$ is a fixed unit vector in $\R^d$ and $p>-1$ then
$$\int_{S^{d-1}}|\langle u,v\rangle|^p\, \Ha^{d-1}(dv)=2\pi^{\frac{d-1}2}\frac{\Gamma\left(\frac{p+1}2\right)}{\Gamma\left(\frac{d+p}2\right)}.$$
If further $q>-1$ and $w$ is another unit vector perpendicular to $u$, we have
$$\int_{S^{d-1}}|\langle u,v\rangle|^p |\langle w,v\rangle|^q\, \Ha^{d-1}(dv)=2\pi^{\frac{d-2}2}\frac{\Gamma\left(\frac{p+1}2\right)\Gamma\left(\frac{q+1}2\right)}
{\Gamma\left(\frac{d+p+q}2\right)}.$$
\end{lemma}

\begin{proof}
We apply the coarea formula for the mapping $g:v\mapsto \langle u,v\rangle^2$ defined on the unit sphere. Since
$$J_1g(v)=2\langle u,v\rangle\sqrt{1-\langle u,v\rangle^2},\quad v\in S^{d-1},$$
and since for $0<r<1$, $g^{-1}\{r\}$ is consists of two spheres of dimension $d-2$ and radius $\sqrt{1-r}$, we get
\begin{align*}
\int_{S^{d-1}}|\langle u,v\rangle|^p\, \Ha^{d-1}(dv)&= \int_0^1\frac 12 r^{\frac{p-1}2}(1-r)^{-\frac 12}2(1-r)^{\frac{d-2}2}\omega_{d-2}\, dr\\
&=\int_0^1 r^{\frac{p-1}2}(1-r)^{\frac{d-3}2}\, dr \,\omega_{d-2}\\
&=\frac{\Gamma\left(\frac{p+1}2\right)\Gamma\left(\frac{d-1}2\right)}
{\Gamma\left(\frac{d+p}2\right)}  \frac{2\pi^{d/2}}{\Gamma\left(\frac{d-1}2\right)}
\end{align*}
and the first assertion follows. 

For the second formula we again use the coarea formula with the mapping $g$ and get
\begin{align*}
\int_{S^{d-1}}&|\langle u,v\rangle|^p |\langle w,v\rangle|^q\, \Ha^{d-1}(dv)\cr
&= \int_0^1\frac 12 r^{\frac{p-1}2}(1-r)^{-\frac 12}
\int_{g^{-1}\{ r\}}|\langle w,v\rangle|^q\,\cH^{d-2}(dv)\, dr.
\end{align*}
The two mappings
$$z\mapsto \pm\sqrt{r}u+\sqrt{1-r}z,\quad z\in S^{d-2},$$
map $S^{d-2}$ injectively onto $g^{-1}\{r\}$, and the area formula yields
\begin{align*}
\int_{g^{-1}\{ r\}}|\langle w,v\rangle|^q\,\cH(dv)&=(1-r)^{\frac{d-2+q}2}\int_{S^{d-2}} |\langle w,z\rangle|^q\, \cH^{d-2}(dz)\\
&=(1-r)^{\frac{d-2+q}2} 2\pi^{\frac{d-2}2}\frac{\Gamma\left(\frac{q+1}2\right)}{\Gamma\left(\frac{d-1+q}2\right)}
\end{align*}
by the first (already proved) equality. The final result follows now from the standard integral
$$\int_0^1 r^{\frac{p-1}2}(1-r)^{-\frac{d-3+q}2}\, dr = \frac{\Gamma\left(\frac{p+1}2\right)
\Gamma\left(\frac{d-1+q}2\right)}{\Gamma\left(\frac{d+p+q}2\right)}.$$
\end{proof}

Next, we give a proof of the combinatorial identity which was used in Section 5.
As usual we define, for $a \in {\mathbb Z}$ and $z \in {\mathbb C}$,
\begin{align*}
{z \choose a} &:=  \frac{z(z-1) \cdots (z-a+1)}{a!} \mbox{ if } a 
\ge 1, \\
{z \choose 0} &:=  1 \mbox{ and } {z \choose a} := 0 \mbox{ if } a 
< 0. 
\end{align*}

\begin{theorem}\label{identity}  Let $d, k, i \in {\mathbb N}_{0}$ and 
$d \ge k-1$. Then
\begin{align} \label{1}
\frac{d+2-k}{d+2} \frac{k+2}{2}& \sum_{m=0}^{k} (-1)^{k-m} {d+1 
\choose k-m}^{-1} \sum_{j=0}^{k} {k+2-j \choose 2}^{-1} \nonumber \\
& \times \sum_{\ell = 0}^{k} {i \choose \ell} {k-i \choose m-\ell} 
{k-i \choose j-\ell} {d-2k + i \choose k + \ell - m - j} = {i \choose 
k}.
\end{align}
\end{theorem}

\begin{proof} The proof is divided into several steps.

(1) For fixed $d, k \in {\mathbb N}_{0}$ and $d \ge k-1$, both sides 
of (\ref{1}) are polynomials in $i$ of degree $k$ at the most. 
Therefore it is sufficient to verify (\ref{1}) for $i = 0,\dots,k$. 
Thus the following has to be shown:

For $d, k \in {\mathbb N}_{0}$ with $d \ge k-1$ and $i \in 
\{0,\dots,k\}$,
\begin{align} \label{2}
\frac{d+2-k}{d+2} \frac{k+2}{2}& \sum_{m=0}^{k} (-1)^{k-m} {d+1 
\choose k-m}^{-1} \sum_{j=0}^{k} {k+2-j \choose 2}^{-1} \nonumber \\
&\times \sum_{\ell = 0}^{k} {i \choose \ell} {k-i \choose m-\ell} 
{k-i \choose j - \ell} {d-2k+i \choose k+\ell - m - j} = \delta_{ik}.
\end{align}

(2) Let $k \in {\mathbb N}_{0}$ and $i \in \{0,\dots,k\}$ be fixed. 
Then the left-hand side of (\ref{2}) is a rational function $q$ in $d$ of 
degree less or equal 0 with possible poles at $-2,\dots,k-3$ (the 
singularity at $d = k-2$ is removable). All poles are simple. Hence 
it is sufficient to check that
\begin{equation}\label{3}
{\rm res}_{d=u} (q) = 0 \mbox{ for } u \in \{-2,\dots,k-3\},
\end{equation}
\begin{equation}\label{4}
\lim_{|d| \to \infty} q(d) = \delta_{ik}.
\end{equation}
Since
\[ {\rm res}_{d=u} \frac{{d - 2k + i \choose k + \ell - m - 
j}}{(d+2){d+1 \choose k-m}} = {u-2k+i \choose k + \ell - m - j} {k - 
m \choose u + 2} (-1)^{k-m-2-u} \]
whenever $u \in {\mathbb Z}$ and $m \in \{0,\dots,k\}$, condition 
(\ref{3}) is equivalent to
\begin{align} \label{5}
\sum_{m=0}^{k}(-1)^{k-m} \sum_{j=0}^{k}{k+2-j \choose 2}^{-1} &
\sum_{\ell=0}^{k}{i \choose \ell} {k-i \choose m-\ell} {k-i \choose 
j-\ell} {u-2k+i \choose k+\ell-m-j} \nonumber \\
&\quad\times {k-m \choose u+2} (-1)^{k-m-2-u} = 0.
\end{align}
Moreover, (\ref{4}) follows as soon as
\begin{equation}\label{6}
\frac{k+2}{2} \sum_{m=0}^{k}(-1)^{k-m} \sum_{
j,\ell=0\atop j = \ell }^{k} {k+2-j \choose 2}^{-1}{i \choose 
\ell} {k-i \choose m-\ell} {k-i \choose j-\ell} = \delta_{ik} 
\end{equation}
has been verified.

(3) First, we establish (\ref{6}). For the left-hand side of 
(\ref{6}), we obtain
\begin{align}\label{7}
\frac{k+2}{2} &\sum_{m=0}^{k}(-1)^{k-m} \sum_{j=0}^{k} {k+2-j 
\choose 2}^{-1} {i \choose j} {k-i \choose m-j} \nonumber \\
&=  \frac{k+2}{2} \sum_{j=0}^{k}{k+2-j \choose 2}^{-1} {i \choose 
j} \sum_{m=0}^{k} (-1)^{k-m} {k-i \choose m-j} \nonumber \\
&= \frac{k+2}{2} \sum_{j=0}^{k}{k+2-j \choose 2}^{-1} {i \choose 
j} \sum_{\overline{m}=0}^{k-j} (-1)^{k-j-\overline{m}} {k-i \choose 
\overline{m}} \nonumber \\
&= \frac{k+2}{2} \sum_{j=0}^{k}{k+2-j \choose 2}^{-1} {i \choose 
j} (-1)^{k-j} (1-1)^{k-i},
\end{align}
where we used that, in case $i \ge j$, the second sum is extended from 
0 to $k-i$ effectively. Now (\ref{6}) clearly follows if $i < k$. If 
$i = k$, then (\ref{7}) equals
\begin{align*}
\frac{k+2}{2} &\sum_{j=0}^{k}{k+2-j \choose 2}^{-1}{k \choose 
j} (-1)^{k-j} \\
&= \frac{k+2}{2} \sum_{j=0}^{k}{k+2 \choose j} {k+2 \choose 
k}^{-1} (-1)^{k-j} \\
&= \frac{k+2}{2}{k+2 \choose k}^{-1} \sum_{j=0}^{k}{k+2 \choose 
j} (-1)^{k-j} \\
&= \frac{k+2}{2} {k+2 \choose k}^{-1} (-1)^{k} {k+2-1 \choose k} 
(-1)^{k} \\
&= 1,
\end{align*}
where
\[ \sum_{j=0}^{k}(-1)^{j} {z \choose j} = (-1)^{k}{z-1 \choose k}, 
\quad z \in {\mathbb C}, k \in {\mathbb N}_{0}, \]
was used (this can be proved by induction with respect to $k$).

(4) We turn to the proof of (\ref{5}). Let $k \in {\mathbb N}_{0}$, 
$i \in \{0,\dots,k\}$ and $u \in \{-2,\dots,k-3\}$. Replacing $m - 
\ell$ by $m$, we see that (\ref{5}) is equivalent to
\begin{equation}\label{8}
\sum_{j=0}^{k}{k + 2 - j \choose 2}^{-1} \sum_{\ell=0}^{k} 
\sum_{m=0}^{k} {i \choose \ell} {k-i \choose m} {k-i \choose j-\ell} {u-2k+i 
\choose k-m-j} {k-m-\ell \choose u+2} = 0.
\end{equation}
Note that ${k-z \choose u+2}$ is a polynomial in $z$ of degree $u + 
2$ and ${z \choose w}$ is a polynomial in $z$ of degree $w$, where $0 
\le w \le u+2$. Hence ${k-z \choose u+2}$ is a linear combination of 
the polynomials ${z \choose w}$, $w = 0,\dots,u+2$.

(5) The preceding discussion demonstrates that (\ref{8}) follows from 
the stronger assertion: For $k \in {\mathbb N}_{0}$, $i \in 
\{0,\dots,k\}$, $u \in \{-2,\dots,k-3\}$ and $w \in \{0,\dots,u+2\}$,
\begin{equation}\label{9}
\sum_{j=0}^{k}{k+2-j \choose 2}^{-1} \underbrace{\sum_{\ell=0}^{k} 
\sum_{m=0}^{k}{i \choose \ell} {k-i \choose m} {k-i \choose j-\ell} 
{u-2k+i \choose k-m-j}{m + \ell \choose w}}_{=: \alpha} = 0.
\end{equation}
Using the relation (Vandermonde convolution)
\[ {m+\ell \choose w} = \sum_{w_{1} + w_{2} = w} {m \choose w_{1}} 
{\ell \choose w_{2}}, \]
we conclude that
\[ \alpha = \sum_{w_{1}+w_{2}=w} \left(\sum_{\ell=0}^{k}{i \choose 
\ell} {k-i \choose j-\ell} {\ell \choose w_{2}}\right) \left( 
\sum_{m=0}^{k}{k-i \choose m}{u-2k+i \choose k-m-j}{m \choose 
w_{1}}\right). \]
The expressions in brackets can be simplified, since we have, for $A \ge B \ge 
0$ and $C, D \in {\mathbb C}$,
\begin{align*}
\sum_{a=0}^{A}{C \choose a}{D \choose A - a}{a \choose B} &= 
\sum_{a=B}^{A}\frac{C \cdots (C - a+1)D \cdots 
(D-A+a+1)a!}{a!(A-a)!B!(a-B)!} \\
%\displaybreak
&= \frac{C \cdots (C-B+1)}{B!} \sum_{a=B}^{A}{C-B \choose a-B}{D 
\choose A - a} \\
&= {C \choose B} \sum_{\overline{a}=0}^{A-B} {C - B \choose 
\overline{a}} {D \choose A - B - \overline{a}} \\
&= {C \choose B}{C - B + D \choose A - B},
\end{align*}
which holds in fact for all $A, B \ge 0$ and $C, D \in {\mathbb C}$. 
Thus
\[ \alpha = \sum_{w_{1}+w_{2}=w} {i \choose w_{2}} {i - w_{2}+k-i 
\choose j - w_{2}} {k-i \choose w_{1}} {k-i-w_{1}+u - 2k+i \choose 
k-j-w_{1}}. \]
The left-hand side of (\ref{9}) now simplifies to
\begin{align*}
 \sum_{w_{1}+w_{2}=w} &{i \choose w_{2}}{k-i \choose w_{1}} 
\sum_{j=0}^{k-w_{1}}{k+2-j \choose 2}^{-1}{k-w_{2} \choose 
j-w_{2}}{u-k-w_{1} \choose k-j-w_{1}} \\
&= \sum_{w_{1}+w_{2}=w} {i \choose w_{2}}{k-i \choose w_{1}} \\
&\quad\times
\underbrace{\sum_{\overline{\j}=0}^{k-w_{1}-w_{2}}{k-w_{2} \choose 
\overline{\j}}{k + 2-w_{2}-\overline{\j} \choose 2}^{-1}{u-k-w_{1} 
\choose k-w_{1}-w_{2}-\overline{\j}}}_{=: \beta}.
\end{align*}
We show that $\beta = 0$. Using the relation
\begin{align*}
{u-k-w_{1} \choose k-w_{1}-w_{2}-\overline{\j}} & =  
(-1)^{k-w_{1}-w_{2}-\overline{\j}}{k-w_{1}-w_{2}-\overline{\j}-1 - 
(u-k-w_{1}) \choose k-w_{1}-w_{2}-\overline{\j}} \\
& =  (-1)^{k-w_{1}-w_{2}-\overline{\j}}{2k-u-w_{2}-\overline{\j}-1 
\choose k-w_{1}-w_{2}-\overline{\j}} \\
& =  (-1)^{k-w_{1}-w_{2}-\overline{\j}}{2k-u-w_{2}- \overline{\j}-1 
\choose k-u+w_{1}-1},
\end{align*}
we get
\begin{align*}
\beta & =  \sum_{\overline{\j}=0}^{k-w_{1}-w_{2}}{k - w_{2} \choose 
\overline{\j}}(-1)^{\overline{\j}}{k+2-w_{2}-\overline{\j} \choose 
2}^{-1}\\
&\hspace{6cm}\times{2k-u-w_{2}-\overline{\j}-1 \choose k-u+w_{1}-1}(-1)^{k-w} \\
& =  \sum_{\overline{\j}=0}^{k-w_{2}}{k-w_{2} \choose 
\overline{\j}}(-1)^{\overline{\j}}\underbrace{{k+2-w_{2}-\overline{\j} 
\choose 2}^{-1} {2k-u-w_{2}-\overline{\j}-1 \choose k-u+w_{1}-1}}_{=: 
\gamma(\overline{\j})} (-1)^{k-w}.
\end{align*}
The last equality follows from the fact that for $\overline{\j} > k - 
w_{1} - w_{2}$, we also have
\[ 2k-u-w_{2}-1 - \overline{\j} < k - u + w_{1} - 1 \]
so that the third binomial coefficient is zero.  The 
denominator of $\gamma$ has simple zeros at $k+2-w_{2}$, $k+1-w_{2}$. Moreover, the 
numerator of $\gamma$ has zeros at $2k-u-w_{2}-1,\dots,k-w_{2}-w_{1}+1$ and
\[ 2k-u-w_{2}-1 \ge k+2-w_{2} \ge k+1-w_{2} \ge k-w_{2}-w_{1}+1 .\]
Hence, $\gamma$ is 
a polynomial in $\overline{\j}$ of degree at most $k-u + w_{1}-3$. 
The degree of $\gamma$ can be estimated by
\[ k-u+w_{1}-3 \le k-u + w - w_{2} - 3 \le k-u + u+2-3 = k-1. \]
By \cite[(5.42)]{GKP},
\[ \sum_{a=0}^{A}{A \choose a}(-1)^{a}p(a) = 0 \]
for any $p \in {\mathbb C}[x]$ with ${\rm deg}(p) \le A-1$, hence 
$\beta = 0$ and therefore $\alpha = 0$. This establishes (\ref{9}).
\end{proof}

\section*{Acknowledgements}
The proof of Theorem \ref{identity} is based on a joint work of DH with Manfred Peter. 

The research of DH and WW has been supported by DFG projects HU 1874/4-2 and WE 1613/2-2 and JR has been partially supported by grant GA\v{C}R P201/10/0472.

\noindent
Author's addresses:

\bigskip

\noindent
Paul Goodey, University of Oklahoma, Department of Mathematics, Norman, OK 73019, U.S.A., email: pgoodey@ou.edu

\bigskip
\noindent
Wolfram Hinderer, 
Robert-Koch-Str. 196, 
D-73760 Ostfildern, Germany, email:
wolfram@hinderer-lang.de

\bigskip

\noindent
Daniel Hug, Karlsruhe Institute of Technology (KIT), 
Department of Mathematics, 
D-76128 Karlsruhe, Germany, email: daniel.hug@kit.edu

\bigskip

\noindent
Jan Rataj, Charles University, Faculty of Mathematics and Physics, 186 75 Praha 8, Czech Republic, email: rataj@karlin.mff.cuni.cz

\bigskip

\noindent
Wolfgang Weil, Karlsruhe Institute of Technology (KIT), 
Department of Mathe\-matics, 
D-76128 Karls\-ruhe, Germany, email: wolfgang.weil@kit.edu


\begin{thebibliography}{99}

\bibitem{AGS} S. Alesker, D. Gourevitch, S. Sahi, On an analytic description of the $\alpha$-cosine transform on real Grassmannians. {\em arXiv:1409.4882v3}, 53 pp. (2015).

\bibitem{Amb1} R.V. Ambartzumian, Combinatorial integral geometry, metrics and zo\-no\-ids. {\em Acta Appl. Math.} {\bf 9}, 3--28 (1987).

\bibitem{Amb2} R.V. Ambartzumian, {\em Factorization Calculus and Geometric Probability}. Cambridge University Press, Cambridge 1990.

\bibitem{GH} P. Goodey, R. Howard, Processes of flats induced by higher dimensional processes. {\em Adv. Math.} {\bf 80}, 92--109 (1990).

\bibitem{GW} P. Goodey, W. Weil, Centrally symmetric convex bodies and Radon transforms on higher order Grassmannians. {\em Mathematika} {\bf 38}, 117--133 (1991).

\bibitem{GW06a} P. Goodey, W. Weil, Directed projection functions of convex bodies. {\it Monatsh. Math.} {\bf 149}, 43--64, 65 (Erratum) (2006).

\bibitem{GW06b} P. Goodey, W. Weil, Determination of convex bodies by directed projection functions. {\it Mathematika} {\bf 53}, 49--69 (2006).

\bibitem{GKP}  R.L. Graham, D.E. Knuth, O. Patashnik, {\em Concrete Mathematics: A Foundation for Computer Science.} 2nd Ed., Addison-Wesley, Reading (Mass.) 1994.  

\bibitem{Hind} W. Hinderer, {\em Integral Representations of Projection Functions}. PhD Thesis, University of Karlsruhe, Karlsruhe 2002. 

\bibitem{HHW} W. Hinderer, D. Hug, W. Weil, Extensions of translation invariant valuations on polytopes. {\em Mathematika} {\bf 61}, 236--258 (2015).

\bibitem{Hug} D. Hug,  {\em Measures, Curvatures and Currents in Convex Geometry.} 
Habilitation Thesis, University of Freiburg, Freiburg 1999. 

\bibitem{HRW} D. Hug, J. Rataj, W. Weil, A product integral representation of mixed volumes of two convex bodies. {\em Adv. Geom.} {\bf 13}, 633--662 (2013).

\bibitem{HSS08} D. Hug, R. Schneider, R. Schuster, Integral geometry of tensor valuations. {\it Adv.\ Appl.\ Math.} {\bf 41}, 482--509 (2008).

\bibitem{HTW} D. Hug, I. T\"urk, W. Weil, Flag measures for convex bodies. In: {\it Asymptotic Geometric Analysis}, 
ed. by M. Ludwig et al., Fields Institute Communications, Vol. {\bf 68}, Springer, 2013, 145--187.

%\bibitem{Kr} R. Kropp, {\em Erweiterte Oberfl\"achenma{\ss}e f\"ur konvexe K\"orper.} Diploma Thesis, University of Karlsruhe, Karlsruhe 1990.


\bibitem{R99} J. Rataj, Translative and kinematic formulae for curvature measures of flat sections. {\em Math. Nachr.} {\bf 197}, 89--101 (1999).


%\bibitem{RZ95} 
%J. Rataj, M. Z\"ahle, Mixed curvature measures
%for sets of positive reach and a translative integral formula.
%{\em Geom. Dedicata} {\bf 57}, 259-283 (1995).

\bibitem{Schn93} R. Schneider, Polytopes and Brunn-Minkowski theory. In:
{\it Polytopes: Abstract, Convex and
Computational (Scarborough 1993)}, ed. by T. Bisztriczky et al.,
NATO ASI Series C, Vol. {\bf 440}, Kluwer, 1994, 273--299.


\bibitem{S} R. Schneider, {\em Convex Bodies: The Brunn-Minkowski Theory}. 2nd Ed., Cambridge University Press, Cambridge 2014.

\bibitem{SW} R. Schneider, W. Weil, {\em Stochastic and Integral Geometry}. Springer, Hei\-del\-berg-New York 2008.

\bibitem{Str} D.W. Stroock, {\em Probability Theory, An Analytic View}. 2nd Ed., Cambridge University Press, Cambridge 2011.
%\bibitem{W1} W. Weil, Zuf\"allige Ber\"uhrung konvexer K\"orper durch $q$-dimensionale Ebenen. {\em Result. Math.} {\bf 4}, 84--101 (1978).

%\bibitem{W2} W. Weil, Kinematic integral formulas for convex bodies. In: {\em Contributions to Geometry,} Proc. Geometry Symp. Siegen 1978 (eds. J. T\"olke \& J.M. Wills). Birkh\"auser, Basel, pp. 60--76 (1979). 
\bibitem{Zieb} I. Ziebarth, {\em Lokales Verhalten konvexer K\"orper und Approximation}. PhD Thesis, Karlsruhe Institute of Technology, Karlsruhe 2014. 

\end{thebibliography}
\end{document}